\documentclass{article}

\newcommand{\bC}{\mathbb{C} \xspace}
\newcommand{\bR}{\mathbb{R} \xspace}
\newcommand{\bQ}{\mathbb{Q} \xspace}
\newcommand{\bF}{\mathbb{F} \xspace}
\newcommand{\Fpbar}{\overline{\mathbb{F}}_{p} \xspace}
\newcommand{\bZ}{\mathbb{Z} \xspace}
\newcommand{\bN}{\mathbb{N} \xspace}
\newcommand{\Qp}{\mathbb{Q}_p \xspace}
\newcommand{\Qpbar}{\overline{\mathbb{Q}}_p \xspace}
\newcommand{\Zp}{\mathbb{Z}_p \xspace}
\newcommand{\Ql}{\mathbb{Q}_{\ell} \xspace}

\newcommand{\frk}{\mathfrak}

\newcommand{\Gal}{\mathrm{Gal} \xspace}

\newcommand{\Fix}{\mathrm{Fix} \xspace}
\newcommand{\Stab}{\mathrm{Stab} \xspace}
\newcommand{\transp}[2]{( #1 \; #2)} 

\newcommand{\ev}{\mathrm{ev} \xspace}
\newcommand{\id}{\mathrm{id} \xspace}
\newcommand{\End}{\mathrm{End} \xspace}
\newcommand{\Hom}{\mathrm{Hom} \xspace}

\newcommand{\Sym}{\mathrm{Sym} \xspace}

\newcommand{\Spec}{\mathrm{Spec} \xspace} 
\newcommand{\CH}{\mathrm{CH} \xspace}

\newcommand{\CHM}{\mathrm{CHM} \xspace}

\newcommand{\bUn}{\mathds 1 \xspace}

\newcounter{todocount} 

\input{packages.tex}
\newtheorem{theorem}{Theorem}[section]
\newtheorem{proposition}[theorem]{Proposition}
\newtheorem{lemma}[theorem]{Lemma}
\newtheorem{conjecture}[theorem]{Conjecture}
\newtheorem{`conjecture'}[theorem]{``Conjecture''}
\newtheorem{corollary}[theorem]{Corollary}

\theoremstyle{definition}
\newtheorem{definition}[theorem]{Definition}

\newtheorem{notation}[theorem]{Notations}

\newtheorem{remark}[theorem]{Remark}

\title{Examples for the standard conjecture of Hodge type}
\author{Thomas AGUGLIARO}
\date{}

\begin{document}
\maketitle

\begin{abstract}
   For each prime number $p$ and each integer $g \geqslant 4$, we construct infinitely many abelian varieties of dimension $g$ over $\Fpbar$ satisfying the standard conjecture of Hodge type. The main tool is a recent theorem of Ancona \cite{anc21}. These varieties are constructed explicitly through Honda--Tate theory. Moreover, they have Tate classes that are not generated by divisors nor liftable to characteristic zero. Also, we prove a result towards a classification of simple abelian varieties for which the result of \cite{anc21} can be applied to. Along the way, we prove results of independent interest about Honda-Tate theory and about multiplicative relations between algebraic integers.
\end{abstract}

\tableofcontents

\section{Introduction} \label{intro}

The goal of this paper is to prove the following result.
\begin{theorem} \label{thm_intro}
	For all prime numbers $p$ and all integers $g \geqslant 4$, there exists infinitely many abelian varieties of dimension $g$ over $\Fpbar$ satisfying the standard conjecture of Hodge type, and with Tate classes which are not Lefschetz and that do not come from specializing Hodge classes of CM-liftings.
\end{theorem}
We will recall the standard conjecture of Hodge type in \Cref{prelim}. As a non-trivial instance, we have the following corollary. Here $\mathcal Z_{\mathrm{num}}^n(X)$ will denote the $\bQ$-vector space of algebraic cycles of codimension $n$ on a variety $X$ modulo numerical equivalence and $\rho_n$ will denote its dimension.
\begin{corollary}
	For all prime numbers $p$ and all even integers $g \geqslant 4$, there exists infinitely many abelian varieties $A$ of dimension $g$ over $\Fpbar$ such that the intersection pairing on $g/2$-cycles
	\begin{align*}
		\mathcal Z_{\mathrm{num}}^{g/2}(A) \times \mathcal Z_{\mathrm{num}}^{g/2}(A) &\to \bQ \\
		\alpha, \beta &\mapsto \alpha \cdot \beta
	\end{align*}
has signature $(s_+, s_-)$ with 
	\begin{align*}
		s_+ + s_- = \rho_{g/2} \; \; \text{ and } \; \;
		s_+ = \rho_0 - \rho_1 + \rho_2 - \rho_3 + \dots + (-1)^{g/2} \rho_{g/2}.
	\end{align*}
\end{corollary}

The standard conjecture of Hodge type was formulated in 1968 by Grothendieck \cite{Gro68}, in a paper in which he proved that together with the other standard conjectures, it would imply the Weil conjecture. But these standard conjectures turned out to be more difficult than the Weil conjecture, which was proven by Deligne in 1974. 

By a specialization argument \cite[Remarque 5.3.2.2]{and04}, if the standard conjecture of Hodge type is proven for all varieties (resp. all abelian varieties) over $\overline \bF_p$ where $p$ is a fixed prime number, then the conjecture for varieties (resp. abelian varieties) over all fields of characteristic $p$ would follow. Moreover, the case of abelian varieties is crucial as the category of motives over $\overline \bF_p$ should be generated by abelian varieties \cite[Proposition 7.3.3.3]{and04}.

When the conjecture was formulated, the evidences were for divisors, as a consequence of a theorem of Segre \cite{seg37}, and in characteristic $0$ by the Hodge--Riemann relations \cite[Theorem 6.33]{voi02}. Since then, the only results on the standard conjecture of Hodge type are rather recent. Milne proved in \cite{mil02} that the conjecture is valid for cycles that are linear combinations of intersections of divisors (so called Lefschetz cycles). In \cite{anc21}, Ancona proved the conjecture for abelian fourfolds, introducing tools to handle some cases with non-Lefschetz cycles. In \cite{kos22}, Koshikawa proves the conjecture for the squares of simple abelian varieties of prime dimensions, using the tool developped by Ancona. Using different techniques, Milne builds examples of non-simple abelian varieties satisfying the conjecture in \cite{mil22}. In \cite{IIK} Ito, Ito and Koshikawa proved the conjecture for self-products of $K3$ surfaces using the Kuga--Satake construction to reduce to the case of abelian varieties. 
	
	 In \Cref{WT}, we highlight a new class of abelian varieties which we will call \textit{mildly exotic abelian varieties}. Roughly speaking, these are abelian varieties with a few more Tate classes than Lefschetz classes, but for which it is still possible to prove the standard conjecture of Hodge type using \cite{anc21}. A first part of the article is devoted to prove that there are many such mildly exotic abelian varieties. 
	\begin{theorem} \label{thm_intro2}
		Mildly exotic abelian varieties over $\overline \bF_p$ satisfy the standard conjecture of Hodge type. Moreover, for each prime number $p$, and each even integer $g \geqslant 4$, there exist infinitely many abelian varieties $A$ over $\overline \bF_p$ of dimension $g$ which are simple and mildly exotic, and such that there are Tate classes in $A$ which do not come from the specialization of a Hodge class of a CM-lifting of $A$.
	\end{theorem}
	The tools for the proof of this theorem are:
	\begin{itemize}
		\item Honda--Tate theory which allows us to construct and study abelian varieties over $\overline \bF_p$.
		\item The theory of motives, which allows us to decompose the cohomology of abelian varieties in factors that are easier to study.
		\item Results of Milne \cite{mil02} and Ancona \cite{anc21} that allow us to study the motives coming from the previous point.
	\end{itemize}
	
	A second goal of the article is devoted to describe more concrete properties of mildy exotic abelian varieties, and classifying these varieties. 
\begin{theorem}
	Let $A$ be a mildly exotic simple abelian variety over $\overline \bF_p$. Then the dimension $g$ of $A$ is even, and there exists a central endomorphism $f$ of $A$ such that $f^2$ is induced by the multiplication by a negative integer.
\end{theorem}
This theorem gives us a better idea of how simple mildy exotic abelian varieties look like. It allows us to build many new interesting examples in the case where $g \equiv 2 \bmod 4$.
\begin{theorem} \label{nc_com-intro}
	For all $g \geqslant 6$ with $g \equiv 2 \bmod 4$ and all prime numbers $p$ there exist
	\begin{itemize}
		\item infinitely many simple abelian varieties $A$ over $\overline \bF_p$ of dimension $g$ whose algebra of endomorphisms is noncommutative.
		\item infinitely many simple abelian varieties $A$ over $\overline \bF_p$ of dimension $g$, whose algebra of endomorphisms is commutative and with at least $4$ independent exotic Tate classes (i.e the $\Ql$-subspace generated by these classes is of dimension $4$ and it does not contain a nonzero Lefschetz class).
	\end{itemize}
\end{theorem}

The proof of \Cref{nc_com-intro} involves the study of multiplicative relations between algebraic integers. This study can be done with two distinct computations for each cases of \Cref{nc_com-intro}, but there is a better explanation using representation theory. The purpose of \Cref{relation-representation} is to give a general result which covers both cases of \Cref{nc_com-intro}. It is of independent interest, and it uses the insight of Girstmair \cite{gir99} that one can extract information from the Galois module structure on the space of multiplicative relations.

The abelian varieties produced in \Cref{main_proof} for the proof of \Cref{thm_intro2} have even dimension $g$ and \textit{Frobenius rank} $g-1$ (for a definition, see \cite[\S 2.1]{DKZ}). In \cite{kos22}, Koshikawa deduced that \cite{anc21} can be applied to prove the standard conjecture of Hodge type for self-product of abelian varieties $A$ of dimension $g$ under the conditions that $g$ is odd and that the Frobenius rank is $g-1$. Moreover, he had examples of varieties satisfying his criterion in prime dimension. It is also possible to prove that the  examples of \Cref{nc_com-intro} produced in \Cref{proof_nc_com} have respectively Frobenius rank $g/2 - 1$ and $g-2$ where $g$ is the dimension of the abelian variety. Hence a finer analysis is needed in order to apply \cite{anc21}.

\subsection{Organization of the paper}
	In \Cref{prelim}, we give an overview of the tools we will be using in the paper, namely we recall the standard conjecture of Hodge type, and we review complex multiplication and motives. \Cref{section_main_result} contains the main result \Cref{main_result} and explains how to deduce \Cref{thm_intro} stated in the introduction from it. In \Cref{weil_numbers}, we compute the smallest field containing a power of a given Weil $q$-number, and we give a proof of \Cref{main_result} in \Cref{main_proof}. In \Cref{WT}, we define and study the structure of mildly exotic abelian varieties. In \Cref{relation-representation} we give a result concerning the multiplicative relations between Weil $q$-numbers, and we use it in \Cref{proof_nc_com} to construct different kinds of mildly exotic varieties (among which simple abelian varieties with noncommutative endomorphism ring).
\subsection{Conventions} \label{conventions}
	In this paper, we will follow the following conventions:
	\begin{enumerate}
		\item If $\mathcal A$ is an abelian scheme over a noetherian base $S$, we will consider the finite dimensional $\bQ$-algebra:
		\[\End^0(\mathcal A):= \End_{S}(\mathcal A) \otimes_{\bZ} \bQ.\]
		\item If $q$ is the power of a prime number $p$, a Weil $q$-number is an algebraic integer $\pi$ such that for all complex embeddings $\sigma: \bQ(\pi) \to \bC$ we have $|\sigma(\pi)| = q^{1/2}$. When we consider valuations on a number field $L$, with a given Weil $q$-number $\pi \in L$, a $p$-adic valuation $v$ will usually be normalized so that $v(q) = 1$.
		\item We will be using smooth Chow motives over a base $S=\Spec(B)$, where $B$ is either a field or the ring of integers $\mathcal O_K$ of a $p$-adic field $K$. We refer to \cite[\S 5.1]{OS11} for generalities about smooth Chow motives (in which they are called Chow motives). Let $F$ be a field of characteristic $0$, we will denote by $\CHM(S)_{F}$ the $F$-linear category of smooth Chow motives over $S$. If $X$ is smooth over $S$, $\mathfrak h(X)$ will be the smooth Chow motive attached to $X$. By a submotive of a motive $M$, we mean a direct summand of $M$.
		If $F'$ is an extension of $F$, and $M$ is in $\CHM(S)_{F}$, we will denote by $M_{F'}$ or $M \otimes_{F} F'$ the motive of $\CHM(S)_{F'}$ obtained from $M$ by extension of scalars.
		If $s: S' \to S$ is a morphism, we will denote by 
		\[s^*: \CHM(S)_F  \to \CHM(S')_F\]
		the functor extending the map $X \mapsto X \times_S S'$ from smooth schemes over $S$ to schemes over $S'$.
	\end{enumerate}
\subsection{Acknowledgements}
	I would like to thank G. Ancona for his many helpful comments during the writing of this article, E. Ambrosi, O. Benoist and F. Fité for useful discussions. I would also like to thank P. Laubie for his advice concerning \LaTeX.
	
	Finally I would like to thank the anonymous referee, who gave me useful suggestions and who motivated me to write \Cref{relation-representation}.
	
	This research was partly supported by the grant ANR--23--CE40--0011 of Agence National de la Recherche.

\section{Preliminaries} \label{prelim}
	We will recall the statement of the \textit{standard conjecture of Hodge type}. Then we will review the necessary notions about complex multiplication, Honda--Tate theory and motives. 

For detailed definitions about algebraic cycles and the standard conjectures, see \cite{and04} and \cite{Kl94}.
Let $X$ be a projective smooth irreducible variety of dimension $n$ over an algebraically closed field $k$. In this section, we fix a prime number $p$.

\begin{notation}
For all $0 \leqslant i \leqslant n$, we denote by $\mathcal Z^i(X)$ the $\bQ$-vector space of $\bQ$-algebraic cycles of codimension $i$, modulo numerical equivalence. That is $\bQ$-linear formal sums of irreducible algebraic subvarieties of $X$ of codimension $i$, such that a formal sum $S$ is identified with $0$ if $S \cdot Z = 0$ for all codimension $n-i$ algebraic subvariety $Z$ of $X$. 

We denote by $\CH^i(X)$ the $i$-th Chow group of $X$, that is the $\bQ$-vector space of $\bQ$-algebraic cycles of codimension $i$ modulo rational equivalence and by
\[ \mathrm{cl}^i: \CH^i(X)\otimes_{\bQ} \Ql \to H_{\ell}^{2i}(X)\]
the cycle class map to $\ell$-adic cohomology for any prime number $\ell$ different from the characteristic of $k$.

\end{notation}
The degree map canonically identifies $\mathcal Z^n(X)$ with $\bQ$.
Moreover, for all $0 \leqslant i \leqslant n/2$, we have a perfect pairing given by the intersection product:
$$\mathcal Z^i(X) \times \mathcal Z^{n-i}(X) \to \mathcal Z^n(X) \cong \mathbb Q.$$
Then for a given hyperplane section $L \in \mathcal Z^1(X)$, we have the \textit{Lefschetz operator}
\begin{align*}
	\mathcal Z^i(X) &\to \mathcal Z^{n-i}(X) \\
	Z &\mapsto Z \cdot L^{n-2i}.
\end{align*}
The \textit{primitive part} $\mathcal P^i(X)$ is defined to be the kernel of $Z \mapsto Z \cdot L^{n-2i+1}$, i.e. the map where one intersects with $L$ once more after the Lefschetz operator. For more context, see  \cite{and04}, \cite{Gro68} or \cite{Kl94}.
\begin{conjecture}[Standard conjecture of Hodge type \cite{Gro68}] \label{SCHT}
	For all $0 \leqslant i \leqslant n/2$, the pairing
	\begin{align*}
	\nu^i: \mathcal P^i(X) \times \mathcal P^i(X) &\longrightarrow \quad \quad \bQ \\
	(Z, Z') \; \quad &\mapsto (-1)^i Z \cdot Z' \cdot L^{n-2i}
	\end{align*}
	is positive definite.
\end{conjecture}
When formulated, this conjecture was known for divisors as a consequence of a theorem of Segre \cite{seg37} and in characteristic $0$ by the Hodge-Riemann relations \cite[Theorem 6.33]{voi02}. We will now recall a recent result concerning a class of cycles for which the conjecture is known.

\begin{definition}[Lefschetz cycles]
	Let us denote by $\mathcal L^*(X)$ the sub $\bQ$-algebra of $\mathcal Z^*(X)$ generated by $\mathcal Z^1(X)$, and by $\mathcal L^i(X)$ the $\bQ$-vector space $\mathcal L^*(X) \cap \mathcal Z^i(X)$ for $0 \leqslant i \leqslant n$. An element of $\mathcal L^*(X)$ is called a Lefschetz cycle.
\end{definition}
The following result has been proved by Milne in \cite[Remark 3.7]{mil02}, and another proof can be found in \cite[Proposition 5.4, Proposition 6.8]{anc21}:
\begin{proposition} \label{SCHT_lefschetz}
	If $X$ is an abelian variety over $k$, the pairing $\nu^i$ defined in \Cref{SCHT} is positive definite when restricted to $\mathcal L^i(X) \cap \mathcal P^i(X)$.
\end{proposition}

\begin{definition}[Tate classes and potential Tate classes]
  	Let $k_0 \subset \overline \bF_p$ be a finite field, $X_0/k_0$ a smooth projective absolutely irreducible variety, $X = X_0\times_{k_0} \overline \bF_p$ and $\ell$ be a prime number different from $p$. For any non-negative integer $i$, a \textit{Tate class} in $H^{2i}(X, \Ql)$ is an eigenvector for the Frobenius endomorphism, associated to the eigenvalue $q^i$. An element $c$ of $H^{2i}(X,\Ql)$ is called a potential Tate class if there is a finite extension $k_1$ of $k_0$ such that $c$ is a Tate class in $H^{2i}(X, \Ql)=H^{2i}(X_1 \times_{k_1} \overline \bF_p)$, where $X_1=X_0 \times_{k_0} {k_1}$. 
	
	The set of potential Tate classes in $H^{2i}(X, \Ql)$ only depends on $X$, so we use the following abuse of notation: a Tate class on $X/\overline \bF_p$ will be a potential Tate class in some $H^{2i}(X, \Ql)$.
\end{definition}

\begin{definition}[Algebraic classes]
	Let $X$ be an irreducible smooth projective variety over $\overline \bF_p$. For any non-negative integer $i$, an algebraic class in $H^{2i}(X, \Ql)$ is an element of the image of the cycle class map
	\[\mathrm{cl}^i: \CH^i(X)_{\Ql} \to H^{2i}(X, \Ql).\]
\end{definition}
\begin{remark}
	For any irreducible smooth projective variety $X$ over $\overline \bF_p$, algebraic classes on $X$ are Tate classes. Tate conjecture predicts that every Tate class on $X$ is also an algebraic class, but few results are known about this conjecture.
\end{remark}
\begin{definition}[Lefschetz and exotic classes]
  Let $X/\overline \bF_p$ be a smooth projective variety, a divisor class in $H^2(X,\Ql)$ is an element of the image of $\mathrm{cl}^1: \CH^1(X)_{\Ql} \to H^2(X,\Ql)$. The subalgebra generated by divisors classes in $H^*(X,\Ql)$ under the cup product will be called the subalgebra of Lefschetz classes $\mathcal L_\ell(X)$. For a non-negative integer $i$, a Lefschetz class in $H^{2i}(X,\Ql)$ is an element of $H^{2i}(X,\Ql) \cap \mathcal L_\ell(X)$.
  Tate classes in $H^{2i}(X, \Ql)$ that are not Lefschetz classes are called \textit{exotic classes}.

If $X_0$ is defined over a finite field $k_0 \subset \overline \bF_p$, we define an exotic class on $X_0$ to be an exotic class on $X = X_0 \times_{k_0} \overline \bF_p$.
\end{definition}
\begin{definition}[CM-fields and CM-algebras]
  Let $K$, $L$ be number fields and $R$ be a ring.
  \begin{itemize}
  \item We say that $K$ is totally real (resp. totally imaginary) if all embedings $K \to \bC$ (resp. no embedings) are real valued. 
  \item We say that $L$ is a CM-field if it is a totally imaginary quadratic extension of a totally real field.
  \item We say that $R$ is a CM-algebra if it is isomorphic to a product of CM-fields.
  \end{itemize}
  \end{definition}
  \begin{remark}
  	If $R$ is a CM-algebra, there exists a unique automorphism $\tau$ of $R$ such that for any ring map
	\[\sigma: R \to \bC, \]
	we have $\tau_{\bC} \circ \sigma = \sigma \circ \tau$, where $\tau_{\bC}$ denotes the complex conjugation.
  \end{remark}
 \begin{definition}[Abelian schemes with CM]
  Let $S$ be an irreducible noetherian scheme and $\mathcal A/S$ be an abelian scheme of relative dimension $g$. We say $\mathcal A$ has CM if there is a CM-algebra $R$ of degree $2g$ over $\bQ$ and an injective map of rings $R \rightarrow \End^0(\mathcal A)$.
 \end{definition}

 \begin{definition}[CM-types] \label{def_CM-types}
 	Let $K$ be a field of characteristic zero and fix an algebraic closure $\overline K$. Let $A$ be an abelian variety over $K$ with complex multiplication by $L$. Let $0$ denote the geometric point corresponding to the zero element of $A$. Then the differential of the action of $L$ on $A$ induces an action on the tangent space of $A$ at $0$.
 	$$ L \rightarrow \End_{\overline K}(T_0 A).$$
 	The subset of characters $\Phi \subset \Hom(L, \overline K)$ of the action of $L$ on $T_0 A$, is called the \textit{CM-type} of $A$ and $L$. \\
	More generally, a CM-type for $L$ will be a subset $\Phi$ of $\Hom(L, \overline K)$ such that we have a partition $\Phi \cup \overline \Phi = \Hom(L, \overline K)$.
 \end{definition}
 \begin{remark}[Link with the Hodge decomposition] \label{hodge_type-CM_type}
 	Let $A$ be a complex abelian variety with complex multiplication by $L$ and type $\Phi$. We have that $\Hom(L, \bC) = \Phi \cup \overline \Phi$. Then the tangent space $V$ of $A$ at $0$ is a complex vector space which decomposes as 
 	$$ V = \bigoplus_{\sigma \in \Phi} V_\sigma $$
 	where $V_\sigma = \{ v \in V| \; \ell \cdot v = \sigma(\ell)v \}$ is an eigenline for the action of $L$. On top of that, \cite[I, \S 1]{mum} gives us the Hodge decomposition:
 	$$H^1(A, \bC) = {V}^{*} \oplus {\overline V}^*,$$
where $^*$ refers to duality and $\overline \cdot$ refers to complex conjugation. \\
 	Hence, the eigenlines for the action of $L$ on $H^1(A)$ are of $2$ kinds: the eigenlines which are in the $H^{1,0}(A)$ corresponds to the characters $\sigma \in \Phi$ and the eigenlines which are in the $H^{0,1}(A)$ corresponds to the characters $\sigma \in \overline \Phi$.
 \end{remark}
 \begin{definition}
 A $p$-adic field is a finite extension $K$ of $\Qp$. We will denote by $\mathcal O_K$ the ring of elements of $K$ which are integral over $\Zp$. The ring $\mathcal O_K$ has a unique maximal ideal $\frk p$, and the residue field $\mathcal O_K/\frk p$ is a finite field.
 \end{definition}
  The following proposition has been used by Shimura and Taniyama in \cite{ST61}, it will provide us information on Weil numbers using information of Hodge theoretic nature. 
 For a proof, see \cite[Lemma 5]{tat71}.
 \begin{proposition}[Shimura-Taniyama's formula] \label{S-T_formula}
 	Let $K$ be a $p$-adic field with integer ring $\mathcal O_K$ whose residue field is of cardinal $q_0$. Fix an algebraic closure $\Qpbar$ of $K$, and an abelian scheme $\mathcal A$ over $\mathcal O_K$ with CM by a field $L$ with CM-type $\Phi$. For each place $v$ of $L$ over $p$, let us denote by $H_v$ the subset of $\Hom(L, \Qpbar)$ of maps which factors through the completion $L \to L_v$, and $\Phi_v = \Phi \cap H_v$. Then there exists $\pi \in L$ which induces the Frobenius of the special fiber of $\mathcal A$ and such that for all place $v | p$ in $L$ we have:
	$$\frac{v(\pi)}{v(q_0)} = \frac{\# \Phi_v}{\#H_v}.$$ 	
 \end{proposition}
       The abelian varieties of this paper will be constructed by the following theorem \cite[Lemma 4]{tat71}, central in Honda--Tate theory.
  \begin{theorem} \label{thm_bourbaki}
    Let $L$ be a CM-field and $\Phi$ be a CM-type for $L$ with values in $\overline{\bQ}_p$. There exist a $p$-adic field $K$, a morphism of $\Qp$-algebras $K \to \overline{\bQ}_p$ and a CM-abelian scheme $\mathcal A$ of CM-type $(L, \Phi)$
    defined over the ring of integers $\mathcal O_K$.
  \end{theorem}
  We will now use the language of motives, for more precisions see \cite{and04} and \cite[\S 5.1]{OS11}. In the following, $S$ will be either $\Spec (C)$ where $C$ is a field, or $\Spec(\mathcal O_K)$ where $\mathcal O_K$ is the ring of integers of a $p$-adic field $K$.
  
 \begin{definition}
 For a given field $F$ of characteristic $0$, the category of relative $F$-linear smooth Chow motives over $S$ is denoted $\CHM(S)_F$. It is a rigid $F$-tensor category, for each relatively smooth scheme $X$ over $S$, we denote by $\frk h(X)$ the smooth Chow motive associated to $X$.
 \end{definition}
 \begin{definition}[Realization functors]
 \begin{itemize}	
 \item For $s$ a point of $S$, $\eta$ the point corresponding to an algebraic closure of $\kappa(s)$ and a prime number $\ell$ different from the characteristic of $\kappa(s)$, we have the $\ell$-adic realization functor
 \[ R_\ell^{(s)}:\CHM(S)_{\bQ} \to \Ql[G]-\text{modules},\] 
 extending $\frk h(X) \mapsto H^*(X_\eta, \Ql)$ endowed with the natural action of $G=\Gal(\kappa(\eta)/\kappa(s))$. 
 \item For $\eta$ a $\bC$-point of $S$, we have the Hodge realization functor with values in the category of $\bQ$-Hodge structures $\text{Hodge}(\bQ)$,
 \[ R^{(\eta)}_B: \CHM(S)_{\bQ} \to \text{Hodge}(\bQ),\]
 extending $\frk h(X) \mapsto H_B^*(X_\eta, \bQ)$ endowed with its natural polarizable Hodge structure.
 \end{itemize}
 \end{definition}
 \begin{remark}
In practice, $S$ will either be $\Spec(k_0)$ with $k_0$ a finite field or $\Spec(\mathcal O_K)$ with $\mathcal O_K$ the ring of integers of a $p$-adic field $K$. In these cases we will take $s$ to be the identity of $\Spec(k_0)$ or the residual point of $\Spec (\mathcal O_K)$, and we will drop the $s$ from the notation $R_\ell$.

In these cases, $s$ is the spectrum of a finite field, hence the Frobenius element is a canonical element of $G$, and we get a canonical automorphism of $R_\ell(M)$ which we will again call the arithmetic Frobenius. We define the geometric Frobenius of $M$ as the inverse of the arithmetic Frobenius.
 \end{remark}
 \begin{definition}
 Let $M$ be a smooth Chow submotive of $\frk h(\mathcal A)$ for an abelian scheme $\mathcal A$ over $S$. We call rank of $M$ and denote by $\dim_{\bQ}(M)$ the $\Ql$-dimension of the $\ell$-adic realization of $M$ for any prime number $\ell$ invertible on $S$. This is independent of $\ell$ by \cite[Corollary 3.5]{jan07}.
 \end{definition}
 
 \begin{definition}[Tate classes in a motive]
 	Let $M$ be a smooth Chow motive over $S$, pure of weight $2i$. For a prime $\ell$ different from $p$, a Tate class in $R_\ell(M)$ is an eigenvector of the Frobenius endomorphism for the eigenvalue $q^i$. 
	
	We will say that $M$ contains a Tate class if there is some $\ell$ such that $R_\ell(M)$ contains a nonzero Tate class (this is independent of $\ell$ by \cite{KM74}). We will say that $M$ contains only Tate classes if the Frobenius endomorphism acts as multiplication by $q^i$ on $R_\ell(M)$.
 \end{definition}
 
 \begin{definition}[Potential Tate classes in a motive]
 	Let $M$ be a smooth Chow motive over $S$, pure of weight $2i$. For $\ell$ different from $p$, let $F$ be the Frobenius endomorphism acting on $R_\ell(M)$. An element $x$ of $R_\ell(M)$ is called a potential Tate class if there exists $k$ such that $x$ is an eigenvector of $F^k$ for the eigenvalue $q^{ik}$ on $R_\ell(M)$.
	
	We will say that $M$ contains only potential Tate classes if there is a power $F^k$ of the Frobenius endomorphism $F$ acting as $q^{ik}$.
 \end{definition}
 
  \begin{definition}
  	Let $M$ be a $\bQ$-linear smooth Chow motive over $S$, of rank $r$. We say $M$ has CM if there is a CM-algebra $R$ of degree $r$ over $\bQ$ and an injective map of rings $R \to \End_{\CHM(S)_{\bQ}}(M)$.
\end{definition}
  \begin{theorem}[Deninger, Murre, O'Sullivan] \label{kunneth-hopf}
	  Let $\mathcal A$ be an abelian scheme of relative dimension $g$ over $S$ which is either the spectrum of a field or $\Spec(\mathcal O_K)$ where $\mathcal O_K$ is the ring of integers a $p$-adic field $K$. The Hopf algebra $\frk h(\mathcal A)$ is symmetric and the component $\frk h^1(\mathcal A)$ of $\frk h(\mathcal A)$ on
which $n_{\frk h(\mathcal A)}$ acts as $n$ for every $n$ satisfies $\bigwedge^{2g+1} \frk h^1(\mathcal A)=0$. In particular we get a decomposition
	\[\mathfrak h(A) = \bigoplus_{n=0}^{2g} \frk h^n(A) \]
where $\frk h^n(A) = \bigwedge^n  \frk h^1(A)$.
Moreover, the realization in $\ell$-adic cohomology of $\frk h^n(A)$ corresponds to $H^n(\mathcal A_{\eta}, \Ql)$ for all $\ell$, where $\eta$ is a geometric point of $S$. This decomposition is compatible with the morphism of specialization \[\Spec(\bF_q) \to \Spec(\mathcal O_K).\]
  \end{theorem}
The decomposition of abelian motives in a relative setting has been done in \cite[Corollary 3.2]{DM91} in the case where $S$ is smooth over a field. For a proof of the previous theorem in a more general setting see O'Sullivan \cite[Theorem 5.1.6]{OS11}.

\begin{remark}
There is a general definition for $\Sym$ objects and $\bigwedge$ objects in a rigid $\otimes$-category. But with this definition, some notation may appear unusual: for an abelian variety $A$ over $S=\Spec(\overline \bF_p)$, we have $\frk h^n(A) = \Sym^n  \frk h^1(A)$. This is because $\frk h^1(A)$ is an odd cohomology group, and because the realization functor $R_\ell$ that we consider does not respect the $\otimes$-structures, it respect them only up to a sign. Hence for this general definition of $\Sym$ and $\bigwedge$, we have as expected
\[ R_\ell (\Sym^n \frk h^1(A)) ) = {\bigwedge}^n R_\ell(\frk h^1(A)) = {\bigwedge}^n H^1(A, \Ql).\]
But for psychological reasons, we will not use this general definition, and what we denote by $\bigwedge^n \frk h^1(A)$ in the previous theorem and in the rest of the article is the motive that is denoted $\Sym^n \frk h^1(A)$ in \cite{OS11}. Hence in this paper, the convention is so that
\[R_\ell \bigg (\bigwedge^n \frk h^1(A)) \bigg ) = \bigwedge^n R_\ell \frk (h^1(A)).\]
\end{remark}

  The complex multiplication of $L$ on an abelian scheme $\mathcal A$ allows us to decompose the motive $\mathfrak h^n(\mathcal A)$ into a direct sum of smaller motives.

\begin{notation} \label{not-decompCM}
	Let $S$ be either the spectrum of a field, or $\Spec(\mathcal O_K)$ where $\mathcal O_K$ is the ring of integers of a $p$-adic field $K$. Let $\mathcal A$ be an abelian scheme of relative dimension $g$ over $S$, with complex multiplication by a CM-field $L$. We will denote by $\tilde L$ a Galois closure of $L$, $G$ the Galois group of $\tilde L$ over $\bQ$, $\alpha$ a primitive element of $L$ over $\bQ$ and $\{\alpha_1, \dots, \alpha_{2g}\}$ the set of conjugates of $\alpha$ in $\tilde L$. We can assume that $\alpha_{i+g} = \overline{\alpha_i}$ for $1 \leqslant i \leqslant g$, and we will identify the set $\{\alpha_1, \dots, \alpha_{2g}\}$ with $\{1, \dots, 2g\}$ to endow the latter with an action of $G$. As an embedding $L \to \tilde L$ is determined by its value at $\alpha$, we have
	\[ \Hom(L, \tilde L) \cong \{ \alpha_1, \dots, \alpha_{2g} \} \cong \{1, \dots, 2g\}. \]
\end{notation}
\begin{lemma} \label{decomp-h1}
With the notations of \ref{not-decompCM}, the motive $\frk h^1(\mathcal A)$, has complex multiplication by $L$. In particular $\frk h^1(\mathcal A)_{\tilde L}$ is canonically the direct sum of rank $1$ submotives indexed by the embeddings $\Hom(L, \tilde L) \cong\{1, \dots, 2g\}$:
	\[ \frk h^1(\mathcal A)_{\tilde L} = \bigoplus_{\sigma \in \Hom(L, \tilde L)} L_\sigma.\]
\end{lemma}
For a proof, see \cite[Proposition 6.6]{anc21}.
\begin{lemma} \label{decomp-hn}
With the notations of \ref{not-decompCM}, for $n$ a nonnegative integer, the motive $\frk h^n(\mathcal A)_{\tilde L}$ has a canonical decomposition as a direct sum of lines
	\begin{equation} \label{eq:dec_lines}
	 	\frk h^n(\mathcal A)_{\tilde L} = \bigoplus_{\substack{I \subset \Hom(L, \tilde L) \\ |I| = n}} L_I,
	 \end{equation}
where $L_I = \bigotimes_{\sigma \in I} L_\sigma$ and the $L_\sigma$ are the lines of \Cref{decomp-h1}.
\end{lemma}
\begin{proof}
	\Cref{kunneth-hopf} gives us $\frk h^n(\mathcal A) = \bigwedge^n \frk h^1(\mathcal A)$, then by multilinearity of $\bigwedge^n$ and by the decomposition provided by \Cref{decomp-h1}, we get the desired result.
\end{proof}

\begin{proposition}  \label{motive_decomposition}
With the notations of \ref{not-decompCM}, we have a decomposition of $\frk h^n(\mathcal A)$ in $\CHM(S)_{\bQ}$
\[ \frk h^n(\mathcal A) = \bigoplus_{\langle I \rangle} M_{\langle I \rangle}, \] 
where $\langle I \rangle$ runs through the set of orbit of $G$ acting on 
\[\big \{I \subset \Hom(L, \tilde L), |I| = n \big \} \cong \big \{I \subset \{1, \dots 2g\}, |I| = n \big \}.\]
More precisely $M_{\langle I \rangle}$ is obtained by descent, by endowing
	 \[(M_{\langle I \rangle})_{\tilde L}:= \bigoplus_{J \in \langle I \rangle} L_J\]
	  with a $\bQ$-structure for all orbits $\langle I \rangle$, where the $L_J$ are the lines of \Cref{decomp-hn} .
\end{proposition}
For a proof, see \cite[Proposition 6.7]{anc21}.
\begin{remark} \label{rmk-rank-eigvalues}
  For a given subset $I$ of $\{1, \dots, 2g\}$, the rank of $M_{\langle I \rangle}$ is $\#\langle I \rangle$. Moreover, these decompositions are stable under pullback by a morphism $S' \to S.$
  
  When $S$ is $\Spec(B)$ where $B$ is either a finite field, or the ring of integers of a $p$-adic field, we have a canonical element $\pi \in \End^0(\mathcal A)$ which we call will call Frobenius.  In case $B$ is a finite field it is the usual geometric Frobenius acting on $\mathcal A$. If $B$ is the ring of integers $\mathcal O_K$ of a $p$-adic field $K$, then by \Cref{S-T_formula} the Frobenius endomorphism of the special fibre lifts to an element $\pi$ of $L \subset \End^0(\mathcal A)$. By functoriality, $\pi$ acts on $\frk h^1(\mathcal A)$, and we can order the set of its eigenvalues $\{\pi_1, \dots, \pi_{2g}\} \subset \tilde L$ so that the Frobenius acts on $M_{\langle I \rangle}$ with eigenvalues in 
  \[\left \{\prod_{j \in J} \pi_j | \; J \in \langle I \rangle \right \}. \]
  Hence, the $\ell$-adic realization of $M_{\langle I \rangle}$ contains a nonzero Tate class if and only if $\#I$ is even and $\prod_{i \in I} \pi_i = q^{\# I/2}$ where $q$ is the cardinal of the residue field of $B$.
\end{remark}

\begin{definition}[Motives containing algebraic classes] \label{def_contain_algcycle}
	Let $M$ be a smooth Chow motive over $S = \Spec(F)$ pure of weight $2n$,  where $F$ is an algebraically closed field. We will say that $M$ contains a nonzero algebraic class if there are maps
	\[c:\bUn(-n) \to M, c': M \to \bUn(-n) \]
	such that $c' \circ c \not = 0$ in $\End(\bUn)\cong \bQ$.
	
	Now, if $M$ is a smooth Chow motive pure of weight $2n$ over $S = \Spec(\mathcal O_K)$ where $\mathcal O_K$ is the ring of integers of a $p$-adic field $K$, or $S = \Spec(k_0)$ with $k_0$ a finite field, we fix a map $s:\Spec(\overline \bF_p) \to S$ and we say that $M$ contains a nonzero algebraic class in its special fibre if $s^*M\in \CHM(\Spec(\overline \bF_p))$ contains a nonzero algebraic class.
\end{definition}

\begin{proposition}
	Let $\mathcal A$ be an abelian scheme over $S = \Spec( \mathcal O_K)$ with $\mathcal O_K$ the ring of integers of a $p$-adic field $K$ or $S=\Spec(k_0)$ with $k_0$ a finite field, and $\Spec (\overline \bF_p) \to S$ be a geometric point of $S$. Let $M$ be a submotive of $\frk h^{2n}( \mathcal A)$ in $\CHM(S)_{\bQ}$, then $M$ contains a nonzero algebraic classes in its special fibre in the sense of \Cref{def_contain_algcycle} if and only if for all $\ell \not = p$, $R_\ell (M) \subset H^{2n}( \mathcal A \times_S \overline \bF_p, \Ql)$ contains a nonzero algebraic class.
\end{proposition}
\begin{proof}
  Denote by $s: \Spec(\overline \bF_p) \to S$ the geometric special point of $S$. By definition, $M$ is a direct summand of $\frk h^{2n}(\mathcal A)$, hence we have an inclusion and a projection $i: M \to \frk h^{2n}(\mathcal A)$, $p: \frk h^{2n}(\mathcal A) \to M$. Assume that $M$ contains a nonzero algebraic class in its special fibre, then we have maps
  \begin{align*}
    c&: \bUn(-n) \to s^*M \\
    c'&: \; s ^* M \; \, \to \bUn(-n)
  \end{align*}
  such that $c' \circ c \not = 0$. Hence for any prime number $\ell \not = p$, $R_\ell(c' \circ c) \not = 0$ so that $R_\ell(c) \not = 0$. Then the image of $R_\ell(i \circ c)$ is by definition in $R_\ell(M)$ and it contains a nonzero algebraic class. 
	
	Now, assume that for all $\ell \not = p$, $R_\ell (M) \subset H^{2n}( \mathcal A \times_S \overline \bF_p, \Ql)$ contains a nonzero algebraic class. Let $\ell \not = p$ be a prime number such that numerical equivalence coincides with $\ell$-adic homological equivalence on $\mathcal A \times_S \overline \bF_p$ (such an $\ell$ exists by \cite[Theorem 1]{clo99}).
        By assumption, there is an algebraic $\ell$-adic class of $\frk h^{2n}(\mathcal A)_{\Ql}$ inside $M_{\Ql}$, which means that we have a map
        \[c_1: \bUn(-n)_{\Ql} \to s^*\frk h^{2n}(\mathcal A)_{\Ql}\]
        such that $R_\ell(c_1) \not = 0$ has values in $R_\ell(M)$. Consider the map
        \[c_2=p \circ c_1: \bUn(-n)_{\Ql} \to s^*M_{\Ql}. \]
        We have
        \begin{align*}
          R_\ell( i ) \circ R_\ell( c_2) &= R_\ell(i \circ p) \circ R_\ell(c_1) \\
          &= R_\ell(c_1),
        \end{align*}
        where the last equality comes from the fact that the image of $R_\ell(c_1)$ is inside $R_\ell(M)$. As $R_\ell(c_1) \not = 0$, we deduce that $R_\ell(c_2) \not = 0$. By compatibility of $\CHM(\overline \bF_p)$ with respect to extensions of coefficients \cite[\S 5]{OS11}, we have an isomorphism
        \[  \Hom(\bUn(-n), s^*M) \otimes_{\bQ} \Ql \xrightarrow{\raisebox{0ex}{\ensuremath{\sim}}} \Hom(\bUn(-n)_{\Ql}, s^*M_{\Ql}). \]
        Hence there exists $c: \bUn(-n) \to s^*M$ such that $R_\ell(c) \not = 0$, so that $R_\ell(i \circ c) \not = 0$. 
        As homological equivalence coincides with numerical equivalence for $\ell$, there exists a map $c_1': s^*\frk h^{2n}(\mathcal A) \to \bUn(-n)$ such that $c'_1 \circ (i \circ c) \not = 0$. Denote by $c'$ the composite
        \[c_1' \circ i : s^*M \to \bUn(-n).\]
        Then $c' \circ c = c_1' \circ i \circ c \not = 0$. This shows that $M$ contains a nonzero algebraic class in its special fibre in the sense of \Cref{def_contain_algcycle}.
\end{proof}
\begin{definition}[Exotic submotives]
	Let $\mathcal A$ be an abelian scheme over $S = \Spec(\mathcal O_K)$ with $\mathcal O_K$ the ring of integers of a $p$-adic field $K$ or $S=\Spec(k_0)$ with $k_0$ a finite field, and $\Spec (\overline \bF_p) \to S$ be a geometric point of $S$. For a submotive $M$ of $\frk h( \mathcal A)$ in $\CHM(S)_{\bQ}$, we will say that $M$ is an \textit{exotic submotive} of $\frk h(\mathcal A)$, or just an exotic motive if $\mathcal A$ is clear in the context, if $M \not = 0$ and the nonzero elements of $R_\ell (M) \subset H^*(\mathcal A \times_S \overline \bF_p, \Ql)$ are exotic Tate classes. Equivalently, $M \not = 0$ is exotic if it is made of Tate classes, and if there is no nonzero Lefschetz classes of $\mathcal A \times_S \overline \bF_p$ in $R_\ell (M)$.
\end{definition}
In this definition, everything is independent of $\ell$ by \cite[Proposition 6.8]{anc21}.
\begin{definition}[Lefschetz submotives]
	Let $\mathcal A$ be an abelian scheme over $S = \Spec(\mathcal O_K)$ with $\mathcal O_K$ the ring of integers of a $p$-adic field $K$ or $S=\Spec(k_0)$ with $k_0$ a finite field, and $\Spec (\overline \bF_p) \to S$ be a geometric point of $S$. For a submotive $M$ of $\frk h( \mathcal A)$ in $\CHM(S)_{\bQ}$, we will say that $M$ contains only Lefschetz classes in its special fibre if every element of $R_\ell(M) \subset H^*(\mathcal A \times_S \overline \bF_p, \Ql)$ is a Lefschetz classe.
\end{definition}

\begin{lemma} \label{exoset_exomot}
	Let $\mathcal A$ be an abelian scheme over the ring of integers $\mathcal O_K$ of a $p$-adic field $K$, or over a finite field $k_0$, with complex multiplication by a CM field $L$ and $\pi \in L$ be the Frobenius of its special fibre. Let $\tilde L$ be a Galois closure of $L$, $\{\pi_1,\dots,\pi_r\}\subset\tilde L$ the set of Galois conjugates of $\pi$ and $I$ be a subset of $\{1, \dots, 2g\}=\Hom(L, \tilde L)$ of cardinal $2n$. Then the motive $M_{\langle I \rangle}$ of \Cref{motive_decomposition} contains a nonzero Lefschetz class in its special fiber if and only if the function
	\begin{align*}
		f_I: \{\pi_1, \dots, \pi_r\} &\longrightarrow \bN \\
          \pi_i &\mapsto \#\{\sigma \in I |\; \exists m, \sigma(\pi)^m = \pi_i^m \}
	\end{align*}
	is invariant under complex conjugation. Moreover it happens if and only if $M_{\langle I \rangle}$ contains only Lefschetz classes.
\end{lemma}
\begin{proof}
	Let $A_0$ be the special fibre of $\mathcal A$ and $A$ be its base change to $\overline \bF_p$. The fact that $M_{\langle I \rangle}$ contains a nonzero Lefschetz class if and only if all of its cohomology classes are Lefschetz is \cite[Proposition 6.8]{anc21}.
	
	Up to replacing $A_0$ by an extension of scalars, we can assume that all divisor classes of $A$ come from divisor classes of $A_0$. Consider
	\[\mathcal L(A)_{\tilde L} = \bigoplus_{\substack{\{\sigma, \sigma'\} \subset \Hom(L, \tilde L) \\ \sigma(\pi)\sigma'(\pi)=q}} L_{\{\sigma, \sigma'\}}\]
	as a submotive of $\frk h^2(A)_{\tilde L}$. The $\bQ$-structure on $\frk h^2(A)_{\tilde L}$ induces a $\bQ$-structure on $\mathcal L(A)_{\tilde L}$, which gives rise to $\mathcal L(A) \in \CHM(k)_{\bQ}$ such that $\mathcal L(A) \otimes_{\bQ} \tilde L = \mathcal L(A)_{\tilde L}$. Denote by $\mathbb L$ the set 
	\[ \big \{ \{\sigma, \sigma'\} \subset \Hom(L, \tilde L) \; | \; \sigma(\pi)\sigma'(\pi) = q  \big\},\]
	and for $\chi: \mathbb L \to \bN$, denote by
	\begin{align*}
		|\chi| &= \sum_{s \in \mathbb L} \chi(s) \\
		L_\chi &= \bigotimes_{s \in \mathbb L} L_s^{\otimes \chi(s)}.
	\end{align*}
	Because we have the decomposition of $\mathcal L(A)_{\tilde L}$ as a sum of lines, we deduce
	\[ \Sym^n \mathcal L(A)_{\tilde L} = \bigoplus_{\substack{\chi: \mathbb L \to \bN \\ |\chi| = n}} L_\chi.\]
	According to Tate's theorem \cite[Theorem 4]{tat66} the $\ell$-adic realization of $\mathcal L(A)$ is the subspace of divisors classes in $H^2(A, \Ql)$. More precisely, it is the subspace of divisors classes in
        \[H^2(A_0 \times_{k_0} \overline \bF_p, \Ql)\]
        which are defined over $k_0$, but by hypothesis all divisors classes of $A$ are defined over $k_0$. Cup product induces a map
	\[ \Delta: \Sym^n \mathcal L(A) \to \frk h^{2n}(A),\]
	which is compatible with the decompositions in sum of lines.
	The image of $R_{\ell} \Delta$ consists of sums of products of divisors, hence it is the subspace of Lefschetz classes in $H_{\ell}^{2n}(A)$. Hence the submotive $M_{\langle I \rangle}$ of $\frk h^{2n}(A)$ contains a nonzero Lefschetz class if and only if it has a nonzero intersection with the image of $\Delta$. By compatibility of $\Delta$ with respect to the decompositions and by Galois equivariance, this is the case if and only if the line $L_I$ of \Cref{decomp-hn} is in the image of $\Delta$. This happens if and only if there is a decomposition
	\[ I = \bigcup_{\{\sigma, \sigma'\} \in J} \{\sigma, \sigma' \}.\]
	Such a decomposition exists if and only if $f_I$ is invariant under complex conjugation.
\end{proof}

\begin{remark} \label{SCHT_reduction}
	The decomposition provided by \Cref{motive_decomposition} is not orthogonal for the intersection product. But we can introduce another product $\langle,\rangle_{1, \rm mot}^{\otimes 2n}$, for which the decomposition is orthogonal \cite[Proposition 6.7]{anc21}. Moreover, we know how to compare these two pairings on $\mathcal P^n(A)$: they are equal up to multiplication by a positive scalar \cite[Lemma 5.3]{anc21}. Hence, it is sufficient to prove the standard conjecture of Hodge type for $M_{\langle I \rangle}$'s separatly, for more details see \cite[\S 5-\S 6]{anc21}.
\end{remark}
   Knowing the rank of the motives of \Cref{motive_decomposition} will be useful in view of the next theorem due to Ancona \cite[Theorem 8.1]{anc21}.
  \begin{theorem} \label{thm_ancona}
    Let $M$ be a relative smooth Chow motive over the ring of integers $\mathcal O_K$ of a $p$-adic field $K$, with a quadratic form
    $q$. Denote by $V_B$ the Betti realization, $V_Z$ the vector space of algebraic classes in the geometric special fibre of $M$ modulo numerical equivalence. Assume that
    \begin{enumerate}
    \item $\dim_{\bQ} V_B = \dim_{\bQ}V_Z = 2$
    \item $q_B:V_B \times V_B \to \bQ$ is a polarization of Hodge structures.
    \end{enumerate}
    Then, $q_Z$ is positive definite.
  \end{theorem}
 
\section{Main result and consequences} \label{section_main_result}
In this section, we state the main result of the paper, \Cref{main_result}. Then we explain how to deduce \Cref{thm_intro} announced in the introduction from it. The proof of \Cref{main_result} is given in \Cref{main_proof}.\\
We fix a prime number $p$.
\begin{theorem} \label{main_result}
  Let $Q$ be an imaginary quadratic number field in which $p$ is inert.
  Then for each totally real field $R$ of even degree $g$ satisfying
  \begin{enumerate}
  \item $p$ is inert in $R$,
  \item If $\tilde R$ is a Galois closure of $R$, then $\Gal(\tilde R/\bQ) = S_g$,
  \item There is an embedding of fields $Q_p = Q \otimes_{\bQ} \Qp \hookrightarrow R_p = R \otimes_{\bQ} \Qp$,
  \end{enumerate}
  there exists a simple abelian variety $A$ over $\overline{\bF}_p$ such that:
  \begin{itemize}
  \item $A$ satisfies the standard conjecture of Hodge type,
  \item $A \times E$ satisfies the standard conjecture of Hodge type, where $E$ is a supersingular elliptic curve,
  \item $\End^0(A)= Q \cdot R$,
  \item there is a Tate class on $A$ that cannot be the reduction of an algebraic class on a CM integral model of $A$, nor a Lefschetz class.
  \end{itemize}
\end{theorem}

The proof will be given in \Cref{main_proof}.

\begin{proposition} \label{fields_construction}
  Let $Q$ be an imaginary quadratic field, in which $p$ is
  irreducible and $g\geqslant 4$ be an even integer. Then there exists infinitely many fields $R$ of degree $g$ satisfying the  assumptions of \Cref{main_result}.
\end{proposition}
\begin{proof}
  This is a classical consequence of the approximation theorem. Let us choose an unramified extension of $Q_p$ of degree $g/2$, that we call $R_p$. Let $P_p$ be the minimal polynomial of any primitive element of $R_p$ over $\Qp$. \\
  Let $\ell, \ell' > g$ be distinct prime numbers. We can consider unitary polynomials of degree $g$, $P_\ell$ and $P_{\ell'}$ such that $P_\ell$ is irreducible over $\bF_\ell$ and $P_{\ell'}$ has exactly $g-2$ distinct roots in $\bF_{\ell'}$. \\
  We can also consider a unitary real polynomial $P_\bR$ of degree $g$ having exactly $g$ distinct real roots. \\
  The approximation theorem gives a unitary polynomial $P$ of degree $g$ which is a good approximation of the polynomials constructed previously at $p$, $\ell$, $\ell'$ and at infinity. Let $R$ be a rupture field of $P$, it satisfies conditions $1$ and $3$ because of the $p$-adic conditions. It satisfies condition $2$ because the Galois group of its normal closure is a subgroup of $S_g$ that contains a $g$-cycle because of the $\ell$-adic condition, and a $2$-cycle because of the $\ell'$-adic condition. And these two cycles generate $S_g$. The field $R$ is totally real because $P$ is split over $\bR$ by the condition at infinity.
\end{proof}

\begin{corollary} \label{even_case}
  For each prime $p$ and even integer $g \geqslant 4$, there exist infinitely many simple abelian varieties $A/\overline{\bF}_p$ of dimension $g$ that satisfy the standard conjecture of Hodge type and such that there exists a Tate class on $A$ which is not a Lefschetz class nor liftable as a Hodge class in characteristic zero.
\end{corollary}
\begin{proof}
  Let $p$ be a prime number, $g \geqslant 4$ an even integer, and be $Q$ an imaginary quadratic field in which $p$ is inert. Let $(R_i)_{i \in \bN}$ be an infinite familly of distinct fields of degree $g$ satisfying the assumptions of \Cref{main_result}, constructed in \Cref{fields_construction}. Then consider the infinite family of fields $L_i = Q R_i$. If $L_i=L_j$, then their maximal totally real subfields are equal, hence $R_i = R_j$ and $i=j$.

  Now, consider $A_i/\overline \bF_p$ the abelian variety constructed in \Cref{main_result} associated to $Q$ and $R_i$. The elements of the family $(A_i)_{i \in \bN}$ are distinct because if $A_i$ is isogenous to $A_j$, then $\End^0(A_i) \cong \End^0(A_j)$ which means that $L_i \cong L_j$ so that $i=j$. Moreover, the $A_i$'s satisfy the assumptions of \Cref{even_case} by hypothesis of \Cref{main_result}.
\end{proof}
\begin{corollary}
  For each prime $p$ and odd integer $g \geqslant 5$, there exist infinitely many abelian varieties $A/\overline{\bF}_p$ of dimension $g$ that satisfy the standard conjecture of Hodge type and such that there exists a Tate class on $A$ which is not a Lefschetz class nor liftable as a Hodge class in characteristic zero.
\end{corollary}
\begin{proof}
  Let $p$ be a prime number, $g \geqslant 5$ an odd integer, and be $Q$ an imaginary quadratic field in which $p$ is inert. Then $g'=g-1 \geqslant 4$ is even, let $(R_i)_{i \in \bN}$ be an infinite familly of distinct fields of degree $g'$ satisfying the assumptions of \Cref{main_result}, constructed in \Cref{fields_construction}. Then consider the infinite family of fields $L_i = Q R_i$. If $L_i=L_j$, then their maximal totally real subfields are equal, hence $R_i = R_j$ and $i=j$.
  
  Consider $A'_i$ the abelian variety of dimension $g'$ given by \ref{main_result}. Let $E$ be a supersingular elliptic curve, then $A_i = A'_i \times E$ is an abelian variety of dimension $g$, and it satisfies the standard conjecture of Hodge type. Moreover the pullback of a non-liftable Tate class in $H^{g'}(A'_i)$ by the projection $A_i \to A'_i$ is a Tate class in $H^{g'}(A_i)$ which is non-liftable.

  Moreover, if $A_i$ is isogenous to $A_j$, then $A'_i$ is isogenous to $A'_j$. Hence $\End^0(A'_i) \cong \End^0(A'_j)$, which means that $L_i \cong L_j$ and $i = j$.
\end{proof}
\begin{remark}
	In the proof that there is no lift of the Tate class to a CM abelian scheme of characteristic $0$, the complex multiplication plays a role through Shimura-Taniyama formula \ref{S-T_formula}. \\
	But it can also be proved that there is no lift of the Tate class to any abelian scheme of characteristic $0$. Indeed, results of Kisin--Madapusi--Shin about Honda--Tate for Shimura varieties of Hodge type \cite{KMS} reduce the question to the CM case.
\end{remark}

\section{A minimality result about Weil numbers} \label{weil_numbers}

	Let $p$ be a prime number, $L/K$ an extension of number fields, and $x$ an element of $L$. If $x^k$ is in $K$ for some $k$, then the $p$-adic valuations of $x$ in $L$ satisfy certain compatibility conditions. Namely, let $v$ be a $p$-adic valuation on $K$, and $v_1, \dots, v_n$ be the extensions of $v$ to $L$. If $x^k \in K$ then $v_1(x) = v_2(x) = \dots = v_n(x) = v(x)$. The goal of this section is to prove \Cref{weil_numbers_minimality} which is about the converse.
	\begin{definition}
		We will say that $x \in L$ is $p$-potentially in $K$ when for all $p$-adic valuations $v$ of $K$, and all extensions $v_1, v_2$ of $v$ to $L$ we have 
		\[v_1(x) = v_2(x). \]
		To say it differently, the valuation of $x$ is independent of the extension $w$ of $v$ to $L$.
	\end{definition}
 \begin{proposition} \label{weil_numbers_minimality}
	Let $\pi \in L$ be a Weil $q$-number where $q$ is a power of $p$. If $\pi$ is $p$-potentially in $K$, then $\pi^k \in K$ where $k$ is the order of the group of roots of unity in a Galois closure $\tilde L$ of $L$.
\end{proposition}
The proof will be given at the end of the section.
\begin{corollary} \label{easy_case}
	Assume that for all $p$-adic valuation $v$ of $K$, there is a unique extension of $v$ to $L$. Then for all Weil $q$-number $\pi \in L$, there is $k \in \bN_{>0}$ such that $\pi^k \in K$.
\end{corollary}
\begin{remark}
	\Cref{weil_numbers_minimality} tells us that the field $\cap_{k \in \bN_{>0}} \bQ(\pi^k)$ is the smallest subfield $K$ of $L$ such that $\pi$ is $p$-potentially in $K$.
\end{remark}
The following definitions are reminiscent of slope vectors as presented in \cite[\S 3]{DKZ}, but are more functorial.
\begin{definition}
	Let $K$ be a field, denote by $P(K)$ the set of $p$-adic valuations on $K$ that are normalized so that $v(p) = 1$. We define the $\bQ$-vector space $V(K) = \bQ^{P(K)}$ of $\bQ$-valued functions on $P(K)$.
\end{definition}
\begin{definition} \label{valuation_vector}
Let $K$ be a field, we define $v_K: K^\times \to V(K)$ that sends a nonzero element $x$ to $(v(x))_{v \in P(K)}$. 
\end{definition}
\begin{proposition} \label{valuation_functoriality}
	Let $\sigma:K \to L$ be a map of number field (e.g. a field extension $L/K$), the map
	\begin{align*} 
	r_{\sigma}: P(L) &\to P(K) \\
	 v &\mapsto v \circ \sigma, 
	\end{align*}
	is surjective, and the $\bQ$-linear map
	\begin{align*} 
	i_{\sigma}: V(K) &\to V(L) \\
	 f &\mapsto f \circ r_{\sigma},
	\end{align*}
	is injective.
	Moreover the following square is commutative:
	\begin{center}
	 \begin{tikzcd}
	  K^\times \arrow[r, "v_K"] \arrow[d] & V(K) \arrow[d, "i_{\sigma}"] \\
	  L^\times \arrow[r, "v_L"] & V(L).
	 \end{tikzcd}
	\end{center}
\end{proposition}
\begin{proof}
	The surjectivity of $r_{\sigma}$ comes from the standard fact that for all valuation $v$ of $K$, there exists an extension $v'$ of $v$ to $L$. The injectivity of $i_{\sigma}$ follows from the surjectivity of $r_{\sigma}$. The commutativity of the diagram comes from the fact that for each $x\in K$ and any valuation $v$ of $L$ we have $v\circ \sigma(x) = v(\sigma(x))$.
\end{proof}
 \begin{remark}
 	For a number field $L$, and $\tilde L$ its Galois closure, with Galois group $G$ over $\bQ$, $G$ acts via \Cref {valuation_functoriality} on $P(\tilde L)$ and on $V(\tilde L)$. 
 \end{remark}
 The following lemma about Weil number is well known, see \cite[Proposition 3.1]{DKZ} and \cite[Theorem 3.4]{Chi}.
\begin{proposition} \label{injectivity_mod_torsion}
	Let $K$ be a field, the map $v_K$ of \Cref{valuation_vector} is injective on Weil $q$-numbers modulo torsion. More precisely, let $k$ be the order of the group of roots of unity in $K$, then for any Weil $q$-numbers $\pi_1$ and $\pi_2$ such that $v_K(\pi_1) = v_K(\pi_2)$, we have $\pi_1^k=\pi_2^k$.
\end{proposition} 
\begin{proof}
	If $v_K(\pi_1) = v_K(\pi_2)$, then $\frac{\pi_1}{\pi_2}$ is a unit in $K$. Moreover it is of module $1$ for all places at infinity, hence $\frac{\pi_1}{\pi_2}$ is a root of unity by a theorem of Kronecker. We conclude that $\pi_1^k = \pi_2^k$.
\end{proof}

 \begin{proposition} \label{galois_description}
 	Let $L$ be a number field with Galois closure $\tilde L$, $G = \Gal(\tilde L/\bQ)$, $\pi \in L$ be a Weil $q$-number and $F = \bQ(\pi^k)$ where $k$ is the order of the group of roots of unity in $\tilde L$. Consider the vector $v_{\tilde L} (\pi) \in V(\tilde L)$ of \Cref{valuation_vector} and the action of $G$ on $V(\tilde L)$. They satisfy
 	\[\Fix_G(v_{\tilde L}(\pi)) = \Fix_G(F). \]
 \end{proposition}
 \begin{proof}
 	An element $\sigma \in G$ is a map $\sigma: \tilde L \to \tilde L$, hence it induces maps $r_\sigma$ and $i_\sigma$ via \Cref{valuation_functoriality}. By functoriality, we have a commutative diagram:
 	\begin{center}
 \begin{tikzcd}
  \tilde L^\times \arrow[r, "v_{\tilde L}"] \arrow[d, "\sigma"] & V(\tilde L) \arrow[d, "i_{\sigma}"] \\
  \tilde L^\times \arrow[r, "v_{\tilde L}"] & V(\tilde L).
 \end{tikzcd}
\end{center}
 If $\sigma \in \Fix_G(F)$ then $\sigma(\pi^k) = \pi^k$. Hence if we apply $v_{\tilde L}$ to the previous equality, the previous diagram tells us that $i_\sigma(v_{\tilde L}(\pi^k))=v_{\tilde L}(\pi^k)$, hence $\sigma \in \Fix_G(v_{\tilde L}(\pi))$. On the other hand, if $\sigma \in \Fix_G(v_{\tilde L}(\pi))$, then we get $v_{\tilde L}(\sigma(\pi)) = v_{\tilde L}(\pi)$. The algebraic integers $\pi$ and $\sigma(\pi)$ are Weil $q$-numbers, by \Cref{injectivity_mod_torsion} we get $\pi^k = \sigma(\pi^k)$, and $\sigma \in \Fix_G(F)$.
 \end{proof}
 With the last proposition, we can prove \Cref{weil_numbers_minimality} pretty efficiently (one can also give a proof of \Cref{weil_numbers_minimality} by doing a computation with the explicit factorization of $\pi$ as a product of prime ideals).
\begin{proof}[Proof of \Cref{weil_numbers_minimality}]
	The condition that $\pi$ is $p$-potentially in $K$ implies that 
	\[\Fix_G(K) \subset \Fix_G(v_{\tilde L}(\pi))\]
	because $\Fix_G(K) = \Gal(\tilde L/K)$ acts transitively on valuations of $\tilde L$ above any valuation $v$ of $K$. But $\Fix_G(v_{\tilde L}(\pi)) = \Fix_G(F)$ by \Cref{galois_description}, where $F = \bQ(\pi^k)$ and $k$ is the order of the group of roots of unity in $\tilde L$. Hence $\Fix_G(K) \subset \Fix_G(F)$, so that $F \subset K$ by Galois theory. We conclude that $\pi^k \in K$.
\end{proof}
 
\section{Proof of the main Theorem} \label{main_proof}

The goal of this section is to prove \Cref{main_result}. First, in  \Cref{CMtype_construction} we build a CM-field $L$ with a CM-type $\Phi$ satisfying some $p$-adic conditions. Then in \Cref{abelianscheme_construction}, we prove that there exists an abelian scheme $\mathcal A$ over the ring of integers $O_K$ of a $p$-adic field $K$, with complex multiplication by $L$ and CM-type $\Phi$. We study properties of $L$ and $\Phi$ in \Cref{end_computation} and \Cref{gal_computation_2}, then we explain how to deduce properties of $\mathcal A$ in \Cref{hodge_type_computation}. In \Cref{lemma_relations} we obtain a full description of the exotic classes in the $\ell$-adic cohomology of $\mathcal A$. The proof of \Cref{main_result} follows from this description.

  Let $g \geqslant 4$ be an even integer, and $Q$, $R$ be number fields as in \Cref{main_result}. We will start by studying the $p$-adic valuations of $L:=Q \cdot R$. \\
  The field $Q$ has a canonical involution $\tau_Q$ induced by \textit{complex conjugation}.
  It induces an involution $\tau_L$ of $L$ that we will write $x
  \mapsto \overline x$, and it extends to an automorphism of $L \otimes \Qp$.
  \begin{lemma} \label{lemma_ideals}
    There are exactly two prime ideals over $p$ in $L$, they are unramified of residual degree $g$ and they are swapped by
    complex conjugation $\tau_L$.
  \end{lemma}
  \begin{proof}
    As $Q$ is totally imaginary and $R$ is totally real, we have that $L = Q \otimes_{\bQ} R$. Hence
    \begin{align*}
      L \otimes \Qp &= \left (Q \otimes_{\bQ} \Qp \right ) \otimes_{\Qp}
                      \left (R \otimes_{\bQ} \Qp \right ) \\
                    &= Q_p \otimes_{\Qp} R_p \\
                    &= R_p \times R_p
    \end{align*}
    where the second equality comes from the behavior of $Q$ and $L$ at the prime $p$, and the last one comes from the embedding $Q_p \hookrightarrow R_p$ (take a primitive element of $Q_p$, its minimal polynomial splits in $R_p$ and the equality follows from the chinese remainder theorem).  \\
    It can be seen that $\tau_Q$ swaps the factors above, hence we have two unramified prime ideals above $p$ in $L$ and they are swapped by complex conjugation and the residual degrees are both $g$.
  \end{proof}
  \begin{definition}
  	We will denote by $v$ and $\overline v$ the discrete valuations corresponding to the prime ideals of \Cref{lemma_ideals}.
  \end{definition}
  
  \begin{lemma} \label{CMtype_construction}
    There is a CM-type $\Phi$ for $L$ (as defined in \ref{def_CM-types}) such that
    \begin{align*}
      \#\Phi_v &= 1 \\
      \# \Phi_{\overline v} &= g-1
    \end{align*}
    where $H_{v^*} = \{ \sigma:L \to \Qpbar| \; \sigma \textrm{ factors through } L_{v^*}\}$ and
      $\Phi_{v^*} = \Phi \cap H_{v^*}$ for $v^*\in\{v, \overline v\}$.
  \end{lemma}
  \begin{proof}
    Let $n_v=1$ and $n_{\overline v} = g-1$, we have that $n_v,n_{\overline v} \in \bN$ and $n_v + n_{\overline v} =
    g = [L_v:\Qp]$. By a combinatorial result described in \cite[\S 4]{tat71} we can
    find such a CM-type $\Phi$ for $L$.
  \end{proof}
    
  \begin{proposition} \label{abelianscheme_construction}
  There exists a $p$-adic field $K$ and an abelian scheme $\mathcal A$ over the ring of integers $\mathcal O_K$ of $K$, with complex multiplication by $L$ such that the CM-type is $\Phi$ satisfying \Cref{CMtype_construction}. Moreover, there is $\pi \in L$ which induces the Frobenius on the special fiber of $\mathcal A$, and such that
  \begin{align}  \label{Shimura-Taniyama}
    {v(\pi)} &= \frac{ \#\Phi_v}{\#H_v} = \frac{n_v}{g} = \frac{1}{g} \\
    {\overline v(\pi)} &= \frac{ \#\Phi_{\overline
                                              v}}{\#H_{\overline v}} = \frac{n_{\overline
                                              v}}{g} = 1-\frac{1}{g},
  \end{align}
  for the normalization $v(q) = \overline v(q) = 1$ where $q$ is the cardinal of the residue field of $K$.
  \end{proposition}
  \begin{proof}
  	The existence of $K$ and $\mathcal A$ follows from Honda--Tate theory (see \Cref{thm_bourbaki}). The existence and properties of $\pi$ follows from \Cref{S-T_formula} applied to $\mathcal A$ and $\Phi$.
  \end{proof}
  
  We will now turn to the study of the abelian scheme $\mathcal A$ satisfying  \Cref{abelianscheme_construction}. 
  
  \begin{lemma} \label{end_computation}
    For any $\mathcal A$ satisfying \Cref{abelianscheme_construction}, the Frobenius $\pi\in L$ is such that $\bQ(\pi) = L$.
  \end{lemma}
  \begin{proof}
    Let $\tilde R$ be a Galois closure of $R$. We have that $Q$ and $\tilde
    R$ are algebraically disjoint, and $\tilde L:=Q \cdot \tilde{R}$ is a Galois
    closure of $L$. Hence we get an isomorphism
    \begin{equation}
    	G := \Gal(\tilde L/\bQ) \cong \Gal( Q/ \bQ) \times \Gal(\tilde R/\bQ).
	\end{equation}	
    We have that $\Gal( Q/ \bQ)$ is isomorphic to $\mu_2$. Let $\{\mu_1, \dots, \mu_g\} = \Hom(R, \tilde R)$, it follows from the hypothesis of \Cref{main_result} that the action of $\Gal(\tilde R/\bQ)$ on $\Hom(R, \tilde R)$ gives an identification $\Gal(\tilde R/\bQ) \cong S_g$ such that for $\sigma \in S_g$ we have $\sigma(\mu_i) = \mu_{\sigma(i)}$. It follows that we have a computation of the Galois group
    \begin{equation} \label{gal_computation_1}
    	G = \mu_2 \times S_g.
    \end{equation}
   	With this isomorphism, we have that $(\varepsilon, \sigma)$ acts on $\Hom(L, \tilde L) = \Hom(Q, Q) \times \Hom(R, \tilde R)$ as 
   	\[(\varepsilon, \sigma)\cdot(\varepsilon', \mu_i) = (\varepsilon \varepsilon', \mu_{\sigma_i}).\]
    We know that $\pi \not \in Q$ nor in $R$, because
    for all $x\in Q$ or $R$, $v(x)
    = \overline v(x)$ and this is \textit{not the case} for $\pi$ by \eqref{Shimura-Taniyama}. \\
    We have that  $\Fix_G(\{\pi\})$ is a subgroup of $\Gal(\tilde L/\bQ)$ that
    contains $\Fix_G(L)$. Through the identification \eqref{gal_computation_1}, $\Fix_G( L)$ corresponds to $\{0\}\times S_{g-1}$ (we see $S_{g-1}$ as a subgroup of $S_g$ through its action on $\{2, \dots, g\}$). Indeed $(\varepsilon, \sigma)(1, \mu_1) = (1, \mu_1)$ if and only if $\varepsilon = 1$ and $\mu_{\sigma(1)} = \mu_1$ . \\
    Hence $\Fix_G(\{\pi\})$ is a subgroup of $\bZ / 2 \bZ \times S_g$ that contains $\{0\}\times S_{g-1}$ and that does not contain $\bZ/2\bZ$ nor $S_g$. By a group theoretic computation we obtain that $H = \{0\} \times S_{g-1}$, hence that $L = \bQ(\pi).$
  \end{proof}

  \begin{definition} \label{weil_numbers_convention}
    Let $\pi$ be a Weil number satisfying \Cref{abelianscheme_construction}, and $\pi_1, \dots, \pi_g$ be the $Q$-conjugates of $\pi$ in some Galois
    closure $\tilde L \subset \bC $ of $L$. Then, we set $\pi_{i+g} = \overline{\pi_i}$ for $1\leqslant i \leqslant g$. \\
  Then the set $\{\pi_1, \dots, \pi_{2g}  \}$ is the \textit{conjugacy class} of $\pi$ over $\bQ$. Hence, once we identify $\{1, \dots, 2g\}$ with $\{\pi_1, \dots, \pi_{2g}\}$, as in \Cref{motive_decomposition}, we get that the complex conjugation acts on $\{1, \dots, 2g\}$ as $i \mapsto i+g \mod{2g}$.
  \end{definition}
  
  \begin{proposition} \label{gal_computation_2}
  	For any $\mathcal A$ satisfying \Cref{abelianscheme_construction}, there is a unique isomorphism of groups $\mu_2 \times S_g \cong \Gal(\tilde L/\bQ)$ such that for all $1 \leqslant i \leqslant g$, the image of $(\varepsilon, \sigma)$ in $\Gal(\tilde L/\bQ)$ acts on $\pi_i$ by:
	  \begin{align}
	  	(\varepsilon, \sigma) \cdot \pi_i = 
	  	\begin{cases} 
	  		\pi_{\sigma(i)} \text{ if }\varepsilon = 1, \\
	  		\overline{\pi_{\sigma(i)}} \text{ if }\varepsilon = -1.	
	  	\end{cases} \label{gal_action}
	  \end{align}
  \end{proposition}
  \begin{proof}
  	The uniqueness comes from the fact that $\tilde L = \bQ(\pi_1, \dots, \pi_g)$, so that the image of a pair $(\varepsilon, \sigma)$ is determined by its action on the generators. 
  	
  	Let $x = \pi + \overline \pi$, which is a primitive element of $R$ and $x_i=\pi_i + \overline \pi_i$. Then we can order  $\Hom(R, \tilde R) = \{\mu_1, \dots, \mu_g\}$ so that
        \[ \mu_i(x) = x_i.\]
        Then the action of $\Gal(\tilde R/\bQ)$ on $\Hom(R, \tilde R)$ defines an isomorphism $\Gal(\tilde R/ \bQ) \cong S_g$.
        By linear disjointness of $Q$ and $\tilde R$, we have an isomorphism 
  \begin{equation*}
  	\Gal(\tilde L/\bQ) \cong \mu_2\times S_g
  \end{equation*}
  satisfying the required property.
  \end{proof}
  
  \begin{proposition} \label{relation}
  	There is a choice of $\mathcal A$ in \Cref{abelianscheme_construction} such that the elements of the conjugacy class of the Frobenius $\pi \in L$ satisfy the relation:
  	\[ \pi_1 \dots \pi_g = q^{g/2}. \] 
  \end{proposition}
  \begin{proof}
	  Let $K', \mathcal A'$ and $\pi'$ be any choice satisfying \Cref{abelianscheme_construction}. The conjugacy class of $\pi'$ over $Q$ is $\{\pi'_1, \dots, \pi'_g\}$, hence 
	  \[\pi_1' \dots \pi'_g = N_{L/Q}(\pi').\]
	  In particular $N_{L/Q}(\pi') \in Q$ is a Weil $q'^g$-number and as $p$ is irreducible in $Q$, \Cref{easy_case} tells us that there is $w > 0$ such that $N_{L/Q}(\pi')^{w} \in \bQ$. By definition of a Weil number we have
	  \[N_{L/Q}(\pi')^{w} \cdot \overline{N_{L/Q}(\pi')^{w}} = {q'}^{wg},\]
	  so that $N_{L/Q}(\pi')^{2w} = {q'}^{wg}$ because every term in the last equality is a real number. If we set $k=2w$, we can choose $K$ an unramified extension of $K'$ of degree $k$ and set $\mathcal A = \mathcal A' \times_{\mathcal O_{K'}} \mathcal O_K$. Then $\pi={\pi'}^k \in L$ will be the Frobenius endomorphism of the special fiber of $\mathcal A$, and we get the desired equality:
	  \[ \pi_1 \dots \pi_g = q^{g/2}. \] 
	  Let us say moreover that the isomorphism $\Gal(\tilde L/\bQ) \cong \mu_2 \times S_g$ of \Cref{gal_computation_2} is the same for $\pi$ and for $\pi'$, by uniqueness of the isomorphism.
  \end{proof}
  \begin{remark}
  According to \Cref{rmk-rank-eigvalues}, a relation as the one in \Cref{relation} expresses the existence of \textit{Tate classes} in a linear subspace of the $\ell$-adic cohomology $H^g(A,\Ql)$ of the special fibre $A$ of $\mathcal A$.
  \end{remark}
  \begin{definition} \label{def_I0}
	  Let $\mathcal A$ be an abelian scheme satisfying \Cref{relation}, $A$ be its special fiber and $\mathcal B$ be any CM-lift of the special fiber $A$ over the ring of integers $\mathcal O_{K'}$ of a $p$-adic field $K'$. Then $\mathcal B$ has CM by $L$ and we can set 
	  \[I_0 = \{1, \dots, g\}, \]
	  and according to the result of \Cref{motive_decomposition}, we get a direct factor $M_{\langle I_0 \rangle}$ of $H^g(\mathcal B)$ of rank $2$, whose realizations are full of Tate classes.
  \end{definition}
  \begin{remark} \label{determinant_h1}
  	There is another description of the motive $M_{\langle I_0 \rangle}$. We defined it using the relation of \Cref{relation} coming from the norm $N_{L/K}$, which is merely a $Q$-linear \textit{determinant}. We can do it more intrinsically by considering the motive of $Q$-dimension $g$, $\mathfrak h^1(\mathcal B) \in \CHM(\mathcal O_{K'})_Q$ (we look at it with coefficient in $Q$, using the action of the complex multiplication). Then we can consider its \textit{determinant} $\bigwedge^g_Q \mathfrak h^1(\mathcal B)$ which is of weight $g$ and $Q$-dimension $1$. Then by forgetting the $Q$-linear structure on $\bigwedge^g_Q \mathfrak h^1(\mathcal B)$, we get a motive of weight $g$ and rank $2$ in $\CHM(\mathcal O_{K'})_{\bQ}$. Moreover the composite map 
  	$$M_{\langle I_0 \rangle} \hookrightarrow \mathfrak h^g(\mathcal B) \cong {\bigwedge}^g_{\bQ} \mathfrak h^1(\mathcal B) \twoheadrightarrow {\bigwedge}^g_{Q} \mathfrak h^1(\mathcal B)$$
  	 is an isomorphism.
  \end{remark}
  We fix an isomorphism $\bC \cong \Qpbar$, so that we can consider Hodge realizations of motives in $\CHM(\mathcal O_{K})$, where $\mathcal O_K$ is the ring of integers of a $p$-adic field $K$ endowed with a $\Qp$-algebra map $K \to \Qpbar$.
  \begin{proposition} \label{hodge_type_computation}
  	 Let $\mathcal A$ be an abelian scheme over $\mathcal O_K$ satisfying \Cref{relation}, and $A$ its special fibre. For any CM-lifting $\mathcal B$ of $A$ over the ring of integers $\mathcal O_{K'}$ of a $p$-adic field  $K'$, the Hodge realization of $M_{\langle I_0 \rangle }$ given by \Cref{def_I0} is of type 
  	\[(g/2+1,g/2-1), (g/2-1, g/2+1),\]
  	hence it does not contain Hodge classes. In particular, the generic fibre of the motive $M_{\langle I_0 \rangle}$ does not contain any algebraic class.
  \end{proposition}
  \begin{proof}
  The isomorphism $\bC \cong \Qpbar$ endows $\bC$ with a discrete valuation $v$, we can normalize it so that $v(q) = 1$.
	  Let $\Phi'$ be the CM-type of $\mathcal B$, and choose an isometric embedding $K' \to \bC$. From the Shimura-Taniyama formula
	  \eqref{Shimura-Taniyama} we deduce: 
	  $$\text{there exists exactly one complex embedding }\sigma \in \Phi'\text{ of }L\text{ in }\bC \text{ such that }v(\sigma(\pi)) = \frac{1}{g}.$$
	  By \Cref{hodge_type-CM_type}, as $\pi$ is a primitive element for $L$, the condition for an embedding $\sigma: L \rightarrow \bC$ to be in $\Phi$ is equivalent to the condition that the eigenspace of $\pi$ acting on $H^1(\mathcal B_{\bC}, \bC)$ for the eigenvalue $\sigma(\pi)$ is of type $(1,0)$ for the Hodge decomposition. \\
	  Up to renumbering the $\pi_i$'s, we can assume $\sigma(\pi) = \pi_1$. Let $L_{\pi_i}$ be the eigenline associated to the eigenvalue $\pi_i$ for the action of $\pi$ on $H^1(\mathcal B_{\bC}, \bC)$.\\
	  The only way that the relation of \Cref{relation} is satisfied is that (after possibly renaming the $\pi_i$'s)
	  $$v(\pi_1) = \dots = v(\pi_{g/2}) = \frac{1}{g}$$ and
	  $$v(\pi_{g/2+1}) = \dots = v(\pi_g) = 1 - \frac{1}{g}.$$
	  Hence we have that the eigenvalues $\pi_2, \dots, \pi_{g/2}$ give rise to lines in
	  $H^{0,1}$ and that the eigenvalues $\pi_{g/2+2}, \dots, \pi_g$ give rise to lines in
	  $H^{1,0}$. The tensor product of the lines $L_{\pi_i}$ for $1 \leqslant i \leqslant g$ will be a line $L$ in
	  $H^{g/2+1, g/2-1}$, and by its definition in \Cref{motive_decomposition}, the realisation of $M_{\langle I_0 \rangle}$ is the direct sum of $L$ and $\overline L$.
  \end{proof}
  \begin{remark}
  	In \Cref{hodge_type_computation}, we prove that Tate classes in $M_{\langle I_0 \rangle}$  cannot be lifted to algebraic classes on the generic fibre of $\mathcal A$ because they cannot be Hodge classes. It is also possible to prove that those Tate classes cannot be lifted to algebraic classes without referring to Hodge theory: the abelian scheme $\mathcal A$ over $\mathcal O_K$ is in fact the localization of an abelian scheme $\widetilde {\mathcal A}$ over $\mathcal O_{\widetilde K}[S^{-1}]$ where $O_{\widetilde K}$ is the ring of integers of a number field $\widetilde K$ that we can assume Galois over $\bQ$, and $S$ is a finite set of prime numbers. The complex multiplication can also be defined for $\widetilde {\mathcal A}$, and so the motive $M_{\langle I_0 \rangle}$ can be seen as a relative motive over $\mathcal O_{\widetilde K}[S^{-1}]$. 
  	
  	The specialization of $M_{\langle I_0 \rangle}$ at some place of $\widetilde K$ above $p$ is full of Tate classes, but we can prove that for infinitely many prime numbers $\tilde p$, the specializations of $M_{\langle I_0 \rangle}$  to places $\frk p$ above $\tilde p$ does not contain Tate classes, and this will prove that $M_{\langle I_0 \rangle}$ does not bear an algebraic class in characteristic zero. Indeed, by Chebotarev, there exists infinitely many primes $\tilde p$ such that $\tilde p$ splits completely in $\widetilde K/ \bQ$ and we can choose $\ell$ such that $\ell$ splits completely in $L$. Then, we can apply Shimura-Taniyama formula \ref{S-T_formula}, to $\widetilde{\mathcal A}$ at a place $\frk p$ above $\tilde p$ in $K$ and we get that there is $\tilde \pi \in L$ inducing the Frobenius on the special fibre of $\widetilde {\mathcal A}$ at $\frk p$ satisfying a property close to what happens in the proof of \Cref{hodge_type_computation}, namely:
  	\begin{align*}
  		0 &= v(\tilde \pi_1) = v(\tilde \pi_{g/2+2}) = \dots = v(\tilde \pi_{g}) \\
  		1 &= v(\tilde \pi_{g/2+1}) = v(\tilde \pi_{2}) = \dots = v(\tilde \pi_{g/2}).
  	\end{align*}
  	Here $\{\tilde \pi_1, \dots, \tilde \pi_g \}$ denotes the set of $Q$-conjugates of $\tilde \pi$ in $\tilde L$ and $v$ is a $\tilde p$-adic valuation on $\tilde L$ where $\tilde L$ is a Galois closure of $L$. Denote by $\widetilde q$ the size of the residue field of $\widetilde K$ at $\frk p$, $v$ is normalized so that $v(\tilde q) = 1$. The eigenvalues of the Frobenius at $\frk p$ acting on $R_{\ell} M_{\langle I_0 \rangle}$ are conjugated to $\tilde \pi_1 \times \dots \times \tilde \pi_g$, and this number cannot be $\widetilde q^{g/2}$ because $v(\tilde \pi_1 \dots \tilde \pi_g) = g/2-1$ and $v(\widetilde q^{g/2}) = g/2$.
  \end{remark}
  	We will now show that the nonzero Tate classes in realizations of $M_{\langle I_0 \rangle}$ are \textit{exotic Tate classes}, and that all exotic classes in realizations of $\frk h(\mathcal A)$ come from the realizations of $M_{\langle I_0 \rangle}$.

\begin{lemma} \label{lemma_relations}
  Let $\mathcal A$ be an abelian scheme satisfying \Cref{relation} with special fibre $A_0/\bF_q$ and $E_0/\bF_q$ be a supersingular elliptic curve. Then the only exotic submotives of
  \[\frk h^*(A_0 \times E_0) \cong \frk h^*(A_0) \otimes \frk h^*(E_0)\]
  are $M_{\langle I_0 \rangle} \otimes \frk h^0(E_0)$ and $M_{\langle I_0 \rangle} \otimes \frk h^2(E_0)$. Moreover, the only exotic submotive of $\frk h^*(A_0)$ is $M_{\langle I_0 \rangle}$.
\end{lemma}
\begin{proof}  
  We have $1 \in I_0$ and $1+g \not \in I_0$ hence by \Cref{exoset_exomot}, the motive $M_{\langle I_0 \rangle}$ does not contain nonzero Lefschetz classes, hence it is exotic. As $\frk h^0(E_0)$ and $\frk h^2(E_0)$ contain only Lefschetz classes  and are invertible, $M_{\langle I_0 \rangle} \otimes \frk h^0(E_0)$ and $M_{\langle I_0 \rangle} \otimes \frk h^2(E_0)$ are also exotic.
  	
  According to \Cref{motive_decomposition}, $\frk h^*(A_0)$ decomposes as a direct sum of motives $M_{\langle I \rangle}$. If such a summand contains only potential Tate classes then $k:=\# I$ is even and
  \begin{equation*}
    \prod_{i \in  I}\pi_i^e = q^{\frac{ek}{2}}
  \end{equation*}
  for some $e$. Then by \Cref{lemma_mult_relations}, we get that
  \begin{itemize}
  \item either $\langle I \rangle = \langle I_0 \rangle$ so that $M_{\langle I \rangle} = M_{\langle I_0 \rangle}$,
  \item or $I$ can be written as $I = \cup_{i \in I \cap [1, g]} \{i, i + g\}$, which implies that $M_{\langle I \rangle}$ is not exotic by \Cref{exoset_exomot}.
  \end{itemize}
  
  This proves the last sentence of the lemma. Now, we see using Kunneth formula that
  \[\frk h^*(A_0 \times E_0)\]
  decomposes as a direct sum of motives,
  \[M_{\langle I \rangle} \otimes \frk h^n(E_0)\]
  with $0 \leqslant n \leqslant 2$. If such a summand contains only potential Tate classes, then $\# I + n$ is even and we have
  \begin{equation*}
    \left (\prod_{i \in  I}\pi_i^e \right ) \times q^{\frac{en}{2}}= q^{\frac{e(k+n)}{2}}
  \end{equation*}
  for some $e$, where $k = \# I$. Hence
  \begin{equation*}
    \prod_{i \in  I}\pi_i^e = q^{\frac{ek}{2}},
  \end{equation*}
  so that we can apply \Cref{lemma_mult_relations}. We get that $k$ is even, so that $n = 0$ or $2$ and
  \begin{itemize}
  \item either $\langle I \rangle = \langle I_0 \rangle$ so that $M_{\langle I \rangle} = M_{\langle I_0 \rangle}$,
  \item or $I$ can be written as $I = \cup_{i \in I \cap [1, g]} \{i, i + g\}$, which implies that $M_{\langle I \rangle}$ is not exotic by \Cref{exoset_exomot}. As $\frk h^n(E_0)$ has rank $1$ and contains only Lefschetz classes, $M_{\langle I \rangle} \otimes \frk h^n(E_0)$ is not exotic either.
  \end{itemize} 
\end{proof}

\begin{lemma} \label{lemma_mult_relations}
  In the context of \Cref{lemma_relations}, let $I \subset \{1, \dots, 2g\}$ be of cardinal $k$, such that
  \begin{equation} \label{relation1}
    \left (\prod_{i \in  I}\pi_i^e \right ) = q^{\frac{ek}{2}}
  \end{equation}
  for some $e$. Then $k$ is even and either $\langle I \rangle  = \langle I_0 \rangle$ or $I$ can be written as
  \begin{equation*}
    I = \bigcup_{i \in I \cap [1, g]} \{i, i+g\}.
  \end{equation*}
\end{lemma}
\begin{proof}
  If there is $i$ in $I$ such that $i+g \in I$ (i.e. if we can
  factorise the relation by $\pi_i^e \pi_{i+g}^e = q$) then by removing $i$
  and $i+g$ from $I$, we get a smaller set $I'$ satisfying the assumptions of the lemma, and finding a partition for $I'$ amounts to finding one for $I$.

  Hence we can assume that there are no $i \in I$ such that $i+g \in I$. In this case we have
  \begin{align*}
    k = \# I &= \# \left  (\bigcup_{i=1}^{g} I \cap \{ i, i+g \} \right ) \\
             & \leqslant \sum_{i=1}^g \# I \cap \{i, i+g \} \\
             & \leqslant g.
  \end{align*}
  More precisely we have $k=0$ or $k=g$. We will prove this by contradiction, assume $0< k < g$. Then we can find $1 \leqslant j \leqslant g$ such that $j \not \in I$ and $j + g \not \in I$. By changing the representative of $\langle I \rangle$, we can choose $i_0 \in I\cap [1,g]$ (if not possible for $I$ itself, replace it by $\overline I$).  By assumption on $I$ this implies that $i_0 +g \not \in I$.
  
  Let us consider $\sigma$ the automorphism of $\tilde L$ corresponding to
  the element $(1, \transp{i_0}{j})$ through the isomorphism
  of \Cref{gal_computation_2}. Here $\transp{i_0}{j}$ denotes the transposition that swaps
  $i_0$ and $j$. By applying $\sigma$ to \eqref{relation1} we get
  \begin{equation*}
    \sigma \left (\prod_{i \in I} \pi_i^e \right ) = q^\frac{ke}{2},
  \end{equation*}
  so that \eqref{gal_action} gives us:
  \begin{equation}
     \prod_{i \in I\setminus{i_0}\cup\{j\}} \pi_i^e = q^{\frac{ke}{2}}. \label{relation2}
  \end{equation}
  Now, by taking the quotient of \eqref{relation1} by \eqref{relation2}
  we get $\pi_{i_0}^e \pi_j^{-e} = 1$, which is impossible since $\pi^e$ has
  $2g$ distinct conjugates. Hence $k= 0$ or $k = g$, in any case $k$ is even.
  
  If $k=0$, then we have the desired partition.
  
  Now assume $k=g$. By changing the representative of $\langle I \rangle$, we can assume $I \cap [1, g] \not = \emptyset$. The algebraic integers $\{\pi_i\}_{1 \leqslant i \leqslant 2g}$ satisfy the relations:
  \begin{align*}
    \prod_{i \in I} \pi_i^e &= q^{\frac{ge}{2}}, \\
    \prod_{i = 1}^{g} \pi_i^e &= q^{\frac{ge}{2}}.
  \end{align*}
  
  By taking the product of these relations, then by cancelling relations of the form $\pi_i^e \pi_{i+g}^e = q^e$, we get the following:
  \[\prod_{i \in I\cap [1,g]} \pi_i^{2e} = q^s \]
  for some $s$. If $\#I \cap [1, g]=g$, then $I = [1,g]$ and $\langle I \rangle= \langle I_0 \rangle$. Otherwise, we have
  \[0 < \#I \cap [1, g] < g, \]
  and we can apply the previous case with $e'=2e$ to get a contradiction.
\end{proof}

\begin{proof}[Proof of \Cref{main_result}]
	Let $Q$, $R$ be number fields as in \Cref{main_result}, and $L = Q \cdot R$. Denote by $\mathcal A/\mathcal O_K$ the abelian scheme provided by \Cref{abelianscheme_construction}, satisfying \Cref{relation}. It has CM by $L$, as well as its special fibre $A_0$ and its geometric special fibre $A$. We see that $A_0$ is geometrically simple with $\End^0(A) = L$ by applying \Cref{end_computation} to extensions of scalars of $\mathcal A$. 
	
	In view of \Cref{motive_decomposition}, we have a decomposition $\frk h^*(\mathcal A)$ into a direct sum of motives $M_{\langle I \rangle}$. Moreover, for any abelian scheme $\mathcal E/\mathcal O_K$ whose special fibre $E_0$ is a supersingular elliptic curve, we obtain a decomposition of $\frk h^*(\mathcal A \times \mathcal E)$ by tensoring the previous decomposition of $\frk h^*(\mathcal A)$ with $\frk h^*(\mathcal E)$.

        Denote by $E$ the geometric special fibre of $\mathcal E$. According to \Cref{SCHT_reduction}, in order to prove the standard conjecture of Hodge type for $A \times E$ and $A$, it suffices to prove the positivity on each summand of the decomposition separately. Let $M$ be a summand of $\frk h^*(\mathcal A \times \mathcal E)$ or $\frk h^*(\mathcal A)$, containing only potential Tate classes. According to \Cref{lemma_relations}, there are two possibility: 
	\begin{itemize}
		\item either $M$ contains only Lefschetz classes in its special fiber,
		\item or M is an exotic motive among $M_{\langle I_0 \rangle}, M_{\langle I_0 \rangle} \otimes \frk h^0(\mathcal E), M_{\langle I_0 \rangle} \otimes \frk h^2(\mathcal E) $, which are of rank $2$ by \Cref{rmk-rank-eigvalues}.
	\end{itemize}
        The standard conjecture of Hodge type holds for motives containing only Lefschetz classes in their special fibers by \Cref{SCHT_lefschetz}. Hence we can assume that $M$ has rank $2$.

        If algebraic cycles in the special fibre of $M$ are numericaly trivial, then the pairing is $0$ and positivity is automatic. Otherwise, \Cref{thm_ancona} applies to $M$ endowed with the intersection product. Indeed, $M$ has rank $2$, has $\dim_{\bQ}(V_Z) = 2$, and as it comes from an abelian scheme, Riemann's relations tell us that the intersection form induces a polarization on the Betti realization.
        
  Hence the the standard conjecture of Hodge type holds true for $A$ and $A \times E$. Moreover according to \Cref{hodge_type_computation}, there is an exotic Tate class in $A$ which is not liftable to characteristic $0$.
\end{proof}
\section{Abelian varieties of Weil--Tate type} \label{WT}
A crucial input for \Cref{main_result} is \Cref{thm_ancona} in which a condition of \textit{rank $2$} is needed. In this section, we will explain some concrete consequences of this condition. In particular, it will justify why we only have new examples for simple abelian varieties in \textit{even dimension}.
\begin{definition}[Mildly exotic abelian varieties] \label{meav}
	Let $A$ be a geometrically simple abelian variety over a finite field, of dimension $g$. We say $A$ is \textit{mildly exotic} if there exists $L$ a CM-field of degree $2g$ such that $A$ has CM by $L$ and
	\begin{enumerate}
	\item There are exotic classes in the cohomology of $A$ (i.e there is an exotic direct summand in the decomposition of \Cref{motive_decomposition}).
	\item Any direct summand of the form $M_{\langle I \rangle}$ of $\mathfrak h^{2i}(A)$ which is an exotic submotive of $\frk h (A)$ is of rank at most $2$.
	\end{enumerate}
	We will say that a simple abelian variety $A/ \Fpbar$ is mildy exotic if it has a model $A_0$ over a finite field $k_0$ which is mildly exotic in the previous sense.
\end{definition}
\begin{remark} \label{exotic-rank-one}
	An exotic submotive $M_{\langle I \rangle}$ of $\mathfrak h^{2i}(A)$ cannot be of rank $1$. Indeed by \Cref{rmk-rank-eigvalues}, its rank is the size of $\langle I \rangle$. This size can only be $1$ if $I = \emptyset$ or if $I = \{1, \dots, 2g\}$, and in these cases $M_{\langle I \rangle}$ isn't exotic by \Cref{exoset_exomot}.
	
	Hence in \Cref{meav}, it is equivalent to ask that $M_{\langle I \rangle}$ has rank $2$ or to ask that it has rank at most $2$.
\end{remark}
The proof of \Cref{thm_intro} given at the end of \Cref{main_proof} can be separated in two main parts. First, constructing simple abelian varieties which are mildly exotic, then proving the standard conjecture of Hodge type for mildly exotic abelian varieties. To prove \Cref{thm_intro} one must also prove that the exotic classes in consideration do not lift as Hodge classes in characteristic $0$, but we will not consider these questions in the following.
\begin{theorem}
	For each prime $p$ and even integer $g \geqslant 4$, there exists infinitely many mildly exotic abelian varieties of dimension $g$ over $\overline \bF_p$.
\end{theorem}
\begin{proof}
	Let $g \geqslant 4$ be an even integer and $Q$ be a quadratic imaginary number field in which $p$ is irreducible. For a field $R$ satisfying \Cref{fields_construction}, we set $L = Q\cdot R$. Then we choose an abelian scheme  $\mathcal A/ \mathcal O_K$ of relative dimension $g$, satisfying \Cref{relation} where $\mathcal O_K$ is the ring of integers of a $p$-adic field $K$. Let $\pi \in L$ be the lift of the Frobenius endomorphism of the special fibre of $\mathcal A$ satisfying \Cref{S-T_formula}. By \Cref{weil_numbers_convention} and \Cref{gal_computation_2}, we can order the conjugates of $\pi$ in $\tilde L$, 
	\[\{\pi_1, \dots, \pi_{2g}\},\]
in such a way that we can compute the action of the Galois group $G = \Gal(\tilde L/\bQ)$.
	With such a context, \Cref{lemma_relations} tells us that there is a unique exotic submotive in the decomposition of $\frk h^*(\mathcal A)$ provided by \Cref{motive_decomposition}, and it is associated to the set $I_0 = \{1, \dots, g\}$. Moreover, $M_{\langle I_0 \rangle}$ has rank $\#\langle I_0 \rangle$ by \Cref{rmk-rank-eigvalues}. Finally, according to \Cref{gal_computation_2}, $I_0$ has only one conjugate which is 
	\[\{g+1, \dots, 2g\}.\] 
	Hence the special fibre $A$ of $\mathcal A$ is mildly exotic and also its extension $\overline A$ to $\Fpbar$. As $\overline A$ determines $R$, and because we have infinitely many choices for $R$, we produced infinitely many mildly exotic abelian varieties.
\end{proof}
\begin{theorem} \label{CSTH_me}
	If $A$ is a mildy exotic geometrically simple abelian variety, then $A$ satisfies the standard conjecture of Hodge type.
\end{theorem}
\begin{proof}
	Let $A$ be a mildy exotic simple abelian variety of dimension $g$ over $\Fpbar$. Let $A_0$ be a model of $A$ over a finite field $k_0$, which is mildly exotic and has CM by a field $L$. We can assume that all algebraic cycles of $A$ are defined over $k_0$, so that proving the standard conjecture of Hodge type for cycles on $A$ amounts to proving it for cycles on $A_0$. Moreover, if we prove the conjecture for any abelian variety isogenous to $A_0$, then the conjecture for $A_0$ follows. According to \cite[Th\'eor\`eme 2]{tat71}, on a finite extension of $k_0$, $A_0$ becomes isogenous to the reduction of a CM abelian variety in characteristic $0$. That is there exists a $p$-adic field $K$ whose residue field $k_1$ is a finite extension of $k_0$, and an abelian scheme $\mathcal A$ over the ring of integers $\mathcal O_K$ with CM by $L$ such that the special fibre of $\mathcal A$ is isogenous to $A_0 \times_{k_0} k$.
	
	Hence we can assume that $A_0$ is the special fibre of $\mathcal A$ and $k_1=k_0$. Let us consider the decomposition of $\frk h^{2i}(A_0)$ and $\frk h^{2i}(\mathcal A)$ provided by \Cref{motive_decomposition}. According to \Cref{rmk-rank-eigvalues} these decomposition are stable by pullbacks, hence the decomposition of $\frk h^{2i}(A_0)$ is the pullback of the decomposition of $\frk h^{2i}(\mathcal A)$ along the morphism $\Spec(k_0) \to \Spec(\mathcal O_K)$. Using the notations of \ref{not-decompCM}, the assumption that $A_0$ is mildly exotic gives us that the ranks of exotic motives $M_{\langle I \rangle}$ associated to $I \subset \{1, \dots, 2g\}$ in $\CHM(\mathcal O_K)_{\bQ}$ are less than $2$. According to \Cref{SCHT_reduction}, to prove the standard conjecture of Hodge type for $A_0$, it suffices to prove a positivity result for each summand $M_{\langle I \rangle}$ containing algebraic classes. For such a direct summand $M_{\langle I \rangle}$, either it contains only Lefschetz classes, or it is exotic. In the first case the positivity result follows from \Cref{SCHT_lefschetz}. In the second case, $M_{\langle I \rangle}$ is of rank $2$ because $A_0$ is mildly exotic, and the positivity result follows from \Cref{thm_ancona}. Note that $M_{\langle I \rangle}$ cannot be of rank $1$ by \Cref{exotic-rank-one}
\end{proof}
\begin{definition} \label{WTav}
	We will say an abelian variety $A$ over $\bF_q$ is of \textit{Weil--Tate type} if there exists a quadratic subfield $Q$ of $\End^0(A)$ such that the motive
	\[M = {\bigwedge}_Q^g \frk h^1(A)\]
consists of potential Tate classes. We will call $M$ a Weil--Tate submotive of $A$ and cohomology classes in $M$ will be called Weil--Tate classes.
\end{definition}
\begin{remark}
	The previous definition is to link with Weil classes in complex abelian varieties such as described in \cite{MZ98}. There are many results on these Weil classes: Markman in \cite{mar22} proves the Hodge conjecture for many abelian fourfolds of Weil type. In \cite{MZ98} they define Weil classes in presence of CM-fields $L$ of degree $2k$ and give criterions for $\bigwedge_L^{g/k} H^1(A, \bQ)$ to be of type $(\frac{g}{k},\frac{g}{k})$ or to contain exceptional Hodge classes. In \cite[\S 16]{MZ98} they talk about the counterpart for Tate classes, for abelian varieties defined over number fields. \Cref{WTav} differs from the usual Weil classes in two ways:
	\begin{itemize}
	 \item We are dealing with Tate classes instead of Hodge classes.
	 \item We are dealing with abelian varieties over finite fields.
	\end{itemize}
	This is why we decided to use a different name, and to include the name of Tate to the definition.
\end{remark}

\begin{theorem} \label{second_theorem}
	Let $A$ be a geometrically simple mildly exotic abelian variety of dimension $g \geqslant 4$ over a finite field $k_0$, with complex multiplication by a field $L$ containing $F=\bQ(\pi)$ where $\pi$ is the Frobenius of $A$. Then $g$ is even, $A$ is of Weil--Tate type, all exotic submotives of $A$ are Weil--Tate and we have two possibilities:
	\begin{enumerate}
		\item Either $\End^0(A)$ is commutative then $\End^0(A) = L$ and there is a quadratic imaginary subfield $Q$ of $L$, i.e we can write $L = Q \cdot R$ with $R$ a totally real field of even degree $g$. Furthermore, motives with exotic Tate classes among the direct summands of the motive of $A$, are of the form $\det_{Q'}( \mathfrak h^1(A))$ with $Q'$ a quadratic imaginary subfield of $L$.
		\item Or $\End^0(A)$ is noncommutative, then it is a quaternionic algebra and $g/2$ is odd. We also have a description :
		\begin{align*}
			L &= B \cdot R \\
			F &= Q_1 \cdot R
		\end{align*}
		where $B = Q_1 \cdot Q_2$ is an imaginary biquadratic field, $Q_1$ is imaginary quadratic, $Q_2$ real quadratic and $R$ is a totally real field of odd degree $g/2$. Moreover $Q_1$ is the unique quadratic subfield  of $F$, and $\det_{Q_1}(\mathfrak h^1(\mathcal A))$ is the unique exotic submotive of $\mathfrak h(\mathcal A)$.
	\end{enumerate}	 
\end{theorem}
\begin{remark}
	The previous theorem suggests how to construct new examples as in \Cref{main_result}. To construct simple abelian varieties with commutative algebra of endomorphisms, consider $L = Q \cdot R$ with $Q$ quadratic imaginary and $R$ totally real of even degree, this is what is done in \Cref{main_result}. To construct examples with noncommutative algebra of endomorphisms, consider $F = Q_1 \cdot R$ with $Q_1$ an imaginary quadratic field and $R$ a totally real field of odd dimension then set $L = Q_2 \cdot F$ with $Q_2$ real quadratic. Moreover, if $R$ has maximal Galois closure, with some $p$-adic condions on $Q_1$, $Q_2$ and $R$, we can construct a mildly exotic abelian variety $A$ with complex multiplication by $L$.
\end{remark}
	\begin{theorem}[Examples, noncommutative case] \label{nc_case}
		For all $g \geqslant 6$ with $g \equiv 2 \bmod 4$, and all prime numbers $p$, there exists infinitely many simple mildly exotic abelian varieties $A$ over $\overline \bF_p$ of dimension $g$, whose algebra of endomorphisms is noncommutative.
	\end{theorem}

	\begin{theorem}[Examples, two exotic submotives] \label{com_case}
		For all $g \geqslant 6$ with $g \equiv 2 \bmod 4$, and all prime numbers $p$, there exists infinitely many simple mildly exotic abelian varieties $A$ over $\overline \bF_p$ of dimension $g$, whose algebra of endomorphisms is commutative and with two distinct exotic submotives.
	\end{theorem}
	We will give a proof of \Cref{nc_case} and \Cref{com_case} in \Cref{proof_nc_com}.
\begin{proof}[Proof of \Cref{second_theorem}]
	Let $\pi$ be the Frobenius of $A$, according to Tate's theorem \cite{tat66}, $\bQ(\pi)$ is the center of $\End^0(A)$ and $\bQ(\pi) \subset L$. Let $m$ be the degree of the previous extension. We have
	\begin{align*}
		[L:\bQ] &= 2g \\
		[L: \bQ(\pi)] &= m
	\end{align*}
	Let $2g'$ be the degree of $\pi$, we have $2g = 2mg'$. \\
	Let $\tilde L$ be a normal closure of $L$, and $G = \Gal(\tilde L/\bQ)$. By naming $\sigma_1, \dots, \sigma_{2g}$ the embeddings of $L$ in $\tilde L$, we get a transitive action of $G$ on the set $\{1, \dots, 2g \}$. According to \Cref{motive_decomposition}, the direct summands are parametrized by the conjugacy classes of subsets $I \subset \{1, \dots, 2g\}$. Moreover by \Cref{rmk-rank-eigvalues}, the rank of $M_{\langle I \rangle}$ is the cardinal of the orbit. By the first condition of \Cref{meav}, we get that there exists $I_0$ which is of even size $2r \leqslant g$ (if $2r > g$, replace $I_0$ by its complement), and such that $R_{\ell}(M_{\langle I_0 \rangle})$ is of dimension $2$. Hence the orbit of $I_0$ is of cardinal $2$, say it is $\{I_0, \tau I_0\}$ where $\tau$ is the complex conjugation (according to \Cref{exoset_exomot}, $I_0 \not = \tau I_0$ because $M_{\langle I_0 \rangle}$ does not contain nonzero Lefschetz classes).  \\
	By the \Cref{lemma_partition}, we get 
	\begin{align*}
		\#\{1, \dots, 2g\} &\leqslant \# I_0 + \# \left (\tau I_0 \right )\\
		&\leqslant 2 \times 2r.
	\end{align*}
	This gives $g \leqslant 2r$. On the other hand, we had $g \geqslant 2r$, hence $g=2r$ is even. Hence the previous inequalities are equalities, so that $I_0$ and $\tau I_0$ are disjoint.
	
	We can rename $\sigma_1, \dots, \sigma_{2g}$ so that $I_0 = \{1, \dots, g\}$ and $\overline \sigma_i = \sigma_{i+g}$.
	
	 For $0 \leqslant i \leqslant 2g$, let $\pi_i = \sigma_i(\pi)$. In this context $\pi_1, \dots, \pi_{2g}$ are the conjugates of $\pi$ in $L$, given with multiplicity $m$.
	According to \cite[Theorem 2]{tat66}, the commutativity of $\End^0(A)$ is equivalent to $m=1$, so the first case is treated by \Cref{lemma_commutative}, whereas the noncommutative case is the case where $m \geqslant 2$, which is treated by \Cref{lemma_noncommutative}.
\end{proof}
	\begin{lemma} \label{lemma_partition}
		In the context of the proof of \Cref{second_theorem}, for any $I\subset \{1, \dots, 2g\}$ such that $M_{\langle I \rangle}$ is exotic, we have $\{1, \dots, 2g\} = I \cup \tau I$.
	\end{lemma}
	\begin{proof}
	 	As $M_{\langle I \rangle}$ has rank $2$ and $I \not = \tau I$, we get $\langle I \rangle = \{I, \tau I\}$. Fix some $i \in I$, for any $1 \leqslant j \leqslant 2g$, we can choose $\sigma$ such that 
		$$\sigma (i) = j.$$
		For such $j$ and $\sigma$ we have $j \in \sigma I$. As $\sigma I \in \{I, \tau I\}$, we conclude $j \in I \cup \tau I$.
	\end{proof}
	\begin{proposition} \label{lemma_commutative}
		Assume we have an abelian variety $A$ in the same setup as in \Cref{second_theorem}, with $\End^0(A)$ commutative. Then $\End^0(A) = L$, and $L$ contains a quadratic imaginary subfield $Q$. Moreover, the motives $\det_{Q'} \left (\mathfrak h^1(A) \right )$ where $Q'$ is a quadratic imaginary subfield of $L$ are the only submotives of $\mathfrak h(A)$ containing exotic Tate classes.
	\end{proposition}
	\begin{remark}
		In the previous lemma, there might be two quadratic subfields of $L$. That is why there might be more than one exotic submotive of $\mathfrak h(A)$.
	\end{remark}
	\begin{proof}
	Recall that $L$ is a maximal commutative subalgebra of $\End^0(A)$. As $\End^0(A)$ is itself commutative, we have $\End^0(A) = L$. The submotives of $\frk h^*(A)$ with exotic Tate classes are of the form $M_{\langle I' \rangle}$ where 
	\[I' \subset \{1, \dots,2g \}\]
	 is of cardinal $g$. Indeed, if $I'$ has size $s$ is such that $M_{\langle I'\rangle}$ is exotic, then so is its complementary, and by \Cref{lemma_partition} we get $s \leqslant g$ and $2g - s \leqslant g$. 
	
	Hence we have to show for all such $I'$ that there is a quadratic subfield $Q'$ of $L$ such that $M_{\langle I' \rangle} = \det_{Q'} \mathfrak h^1( \mathcal A)$. We will do so by Galois theory: we have $H = \Fix(\pi_1) \subset G$, and we want to construct a subgroup $Z$ of $G$ containing $H$ and of index $2$ inside $G$. Assume $1 \in I'$, and let us consider the transitive action of $G$ on $\{I', \tau I'\}$, and $Z = \Fix(I')$. Because the action is transitive, $Z$ has index $2$ in $G$, and the complex conjugation $\tau$ is not in $Z$. \\
	Hence, there correspond by Galois theory a subfield $Q'$ of $\tilde L$ which is quadratic and imaginary. If $\sigma \in H$, then $\sigma(\pi_1) = \pi_1$, hence $\sigma(I') \cap I' \not = \emptyset$. It follows that $I' = \sigma I'$, which means that $\sigma \in Z$. Hence $H \subset Z$, so that $Q' \subset L$ by Galois theory. Moreover the $Q'$-conjugates of $\pi_1$ are $\{\pi_i| \; i\in I'\}$, so that $M_{\langle I \rangle} = \det_{Q'}(\mathfrak h^1(\mathcal A))$ as in \Cref{determinant_h1}.
	\end{proof}	
	\begin{lemma} \label{lemma_tech}
		In the setup of the proof of \Cref{second_theorem}, we have that 
		\[\{\pi_1, \dots, \pi_g \} \text{ and } \{\pi_{g+1}, \dots, \pi_{2g} \} \] are disjoint.
	\end{lemma}
	\begin{proof}
		Assume there are $1 \leqslant i, j \leqslant g$ such that $\pi_i = \pi_{j+g}$. Then we would have $\overline {\pi_j} = \pi_i$, hence $\pi_i \pi_j = q$. Then, the set $I' = \{1, \dots, g\}\setminus \{i, j\}$ would lead to an exotic $M_{\langle I' \rangle}$, which is impossible by \Cref{lemma_partition} because $I'$ is too small.
	\end{proof}
	\begin{proposition} \label{lemma_noncommutative}
		Let $A$ be an abelian variety with CM by $L$ as in \Cref{second_theorem}. Assume that $\End^0(A)$ is noncommutative. Then $\End^0(A)$ is a quaternionic algebra and $g/2$ is odd. Moreover $F = \bQ(\pi)$ contains a unique imaginary quadratic subfield $Q_1$, and the only exotic summand of $\mathfrak h(\mathcal A)$ is $\det_{Q_1}(\mathfrak h^1(\mathcal A))$. This gives a description of $F$ as $F = Q_1 \cdot R$ where $R$ is a totally real field of odd degree. Moreover, $L$ contains a real quadratic subfield $Q_2$ which is not in $F$, hence $L = B \cdot R$ where $B = Q_1 \cdot Q_2$ is an imaginary biquadratic field.
	\end{proposition}
	\begin{proof}
	Recall that $m=[L:F]$, by \cite[Theorem 2]{tat66} the non-commutativity hypothesis corresponds to $m \geqslant 2$. Let $g' = g/m$, up to renaming the $\sigma_i$'s, by \Cref{lemma_tech} we can assume without modifying $I_0=\{1, \dots, g\}$ that 
	$$\{\pi_1, \dots, \pi_{g'} \} = \{\pi_1, \dots, \pi_g\}.$$
	The relation $\prod_{i \in I_0} \pi_i^e = q^{er}$ is witnessing that $M_{\langle I_0 \rangle }$ contains potential Tate classes. Because we know the multiplicity of the $\pi_i$'s, we get:
	$$\left ( \prod_{i=1}^{g'} \pi_i^e \right )^m = q^{er}.$$
	If $g'$ is even, the relation $\prod_{i=1}^{g'} \pi_i^{em} = q^{er}$ expresses that $M_{\langle I' \rangle }$ is exotic for $I'= \{1, \dots, g'\}$. This is impossible by \Cref{lemma_partition} because $I'$ is too small, hence $g'$ must be odd. By taking an extension of degree $em$ of $k_0$ , we can assume $\prod_{i=1}^{g'} \pi_i = q^{r/m}$.
	
	 By applying \Cref{lemma_big} to $J = \{1, \dots, g'\}$, we get that $g' \geqslant g/2 = g' \frac{m}{2}$. Hence we have $m=2$ and $g/2 = g/m = g'$ is odd.
	 
	We are left to prove that there is a biquadratic subfield of $L$ and to describe the exotic submotives of $\mathfrak h(A)$. We will use Galois theory as in the proof of \Cref{lemma_commutative}. The Galois group $G= \Gal(\tilde L/\bQ)$ acts transitively on $\{1, \dots, 2g\}$. Let $$I' = \{1, \dots, g'\}, I'' = \{g'+1, \dots, g\}$$
	 and $i_1 \in I''$ such that $\pi_1 = \pi_{i_1}$. Then by \Cref{lemma_action2}, the action of $G$ on subsets stabilizes $W = \{I', I'', \overline{I'}, \overline{I''}\}$, and the action is transitive. Then we can consider the subgroups 
	 $$Z = \Fix(I'),\; Z_1 = \Stab(\{I', I''\}) = \Fix(I_0) \text{ and } Z_2 = \Stab(\{I', \overline{I'} \}).$$
	Because the action of $G$ on $W$ is transitive, $Z$ has index $4$ in $G$. Because the action of $Z_1$ on $\{I',I''\}$ is transitive, $Z$ has index $2$ in $Z_1$ and the same for $Z_2$. Hence by Galois theory, there correspond a subfield $B$ of degree $4$ of $\tilde L$ and two disctinct quadratic subfield $Q_1$ and $Q_2$ of $B$. Because $\tau \not \in Z_1$ and $\tau \in Z_2$, we have that $Q_1$ is imaginary and $Q_2$ is real. Moreover $B$ is generated by $Q_1$ and $Q_2$, hence is biquadratic. 
	
	Moreover $F = \bQ(\pi)$ corresponds to the subgroup $H_F=\Stab(\{\sigma_1, \sigma_{i_1}\})$, where $i_1$ is such that $\sigma_{i_1}(\pi) = \sigma_1(\pi)$. If $\sigma \in H_F$ then either $1\in\sigma(I')$ or $i_1 \in \sigma(I')$ depending on the value of $\sigma \circ \sigma_1$. As $G$ stabilizes $W$, we only have two possibilities: $\sigma I' = I'$ or $\sigma I' = I''$. We deduce that $H_F\subset Z_1$ and $Q_1$ is a subfield of $F$. Similarly, we prove that $B$ is a subfield of $L$ and this finishes the description of the fields $L$ and $F$. \\
	By construction, the $Q_1$ conjugates of $\sigma_1$ are $\sigma_1, \dots, \sigma_g$. Hence the relation saying that $M_{\langle I_0 \rangle}$ contains potential Tate classes becomes $N_{L/Q_1}(\pi) = q^{\frac{g}{2}}$. As in \Cref{determinant_h1}, this tells us that $M_{\langle I_0 \rangle}$ is $\det_{Q_1}(\mathfrak h^1(\mathcal A))$. \Cref{lemma_exosets} tells us that $M_{\langle I_0 \rangle}$ is the only exotic summand of $\mathfrak h(\mathcal A)$ and this concludes the proof.
	\end{proof}
	\begin{lemma} \label{lemma_big}
		If we are in the setting of \Cref{second_theorem} with $I_0 = \{1, \dots, g\}$. For all $J \subset I_0$ such that $\prod_{j \in J} \pi_j = q^{s/2}$ for some $s$, and for any $\sigma \in G$ such that $\sigma J \not = J$ and $\sigma J \subset I_0$, we have $\{1, \dots, g\} = J \cup \sigma J$. In particular if $J \not = \emptyset$ is such that $\prod_{j \in J} \pi_j = q^{s/2}$ then $\#J \geqslant g/2$.
	\end{lemma}
	\begin{proof}
		Because $\sigma J \not = J$, we have that $J$ is not empty. We deduce that $\sigma J \subset I_0 \cap \sigma I_0$ is also non-empty, hence $\sigma I_0 \not = \overline I_0$ and $\sigma I_0 = I_0$. Furthermore, we have $J \cap \overline {\sigma J} = \emptyset$, so that $S = J \cup \overline {\sigma J}$ is of cardinal $2\#J$. We can also remove the duplicates: 
		$$S' = \left (J \setminus \sigma J \right ) \cup \left (\overline {\sigma J} \setminus \overline J \right ).$$
		The set $S'$ is non-empty because $J\not = \sigma J$ and of cardinal $2\#J - 2 \# (J\cap \sigma J)$. Moreover, we have:
		\begin{align*}
			\prod_{i\in S'} \pi_i &= \prod_{i \in J\setminus \sigma J} \pi_i \prod_{i \in \overline {\sigma J}\setminus \overline J} \pi_i \\
			&=\prod_{i \in J} \pi_i \left (\prod_{i \in J \cap \sigma J} \pi_i \right )^{-1} \prod_{i \in \overline {\sigma J}} \pi_i \left ( \prod_{i \in \overline {\sigma J} \cap \overline J} \pi_i \right )^{-1} \\
			&= q^{s/2} \left (\prod_{i \in J \cap \sigma J} \pi_i \right )^{-1} (\sigma \cdot \overline {q^{s/2}}) \left ( \prod_{i \in \sigma J \cap J} \overline {\pi_i } \right )^{-1} \\
			&= q^s \left ( \prod_{i \in \sigma J \cap J} \pi_i \overline {\pi_i } \right )^{-1} \\
			&= q^{s - \#(J \cap \sigma J)}.
		\end{align*}
		As $S' \not = \emptyset$, $M_{\langle S' \rangle}$ is a summand of $\mathfrak h^{2s - 2\#(J \cap gJ)}(\mathcal A)$ which contains exotic Tate classes, and by \Cref{lemma_partition} this possible only if $\{1, \dots, 2g\} = S' \cup \overline{S'}$, which implies $\{1, \dots, 2g\} = S \cup \overline{S}$ because $S' \subset S$. So that we have
		$$\{1, \dots, 2g\} = J \cup gJ \cup \overline{J} \cup \overline {gJ}.$$
		As $\{1, \dots, g\}$ is disjoint from $\overline J$ and $\overline {gJ}$, by intersecting the previous equality with $\{1, \dots, g\}$, we get $\{1, \dots, g\} = J \cup \sigma J$.
	\end{proof}
	\begin{lemma} \label{lemma_action2}
		In the setting of \Cref{lemma_noncommutative}, let $I' = \{1, \dots, g'\}$ and $I'' = \{g'+1, \dots, g\}$. Then $G=\Gal(\tilde L/\bQ)$ stabilizes $W = \{I', I'', \overline{I'}, \overline{I''}\}$.
	\end{lemma}
	\begin{proof}
		Recall that $I_0 = \{1, \dots, g\}$, let $J \in W$, we can assume $J =I'$ or $J = I''$. Let $\sigma \in G$, we have either $\sigma I_0 = I_0$ or $(\tau \sigma )I_0 = I_0$. By changing $\sigma$ by $\tau \sigma$ if necessary, we can assume that $\sigma I_0 = I_0$. \\
		If $\sigma J = J$ then there is nothing to prove. Otherwise $J \subset I$ and $J \not = \sigma J$, hence we can apply \Cref{lemma_big} to deduce that $\{1, \dots g\} = J \cup \sigma J$. Because $J$ and $\sigma J$ are of cardinal $g/2$, we deduce that $J$ and $\sigma J$ are disjoint. Hence $\sigma J$ is the complementary of $J$ in $\{1, \dots, g\}$, that is to say $\sigma J \in \{I', I''\}$.
	\end{proof}
	\begin{lemma} \label{lemma_exosets}
		In the setting of \Cref{lemma_noncommutative} with $I_0 = \{1, \dots g\}$, let $J \subset \{1, \dots, 2g\}$ be a subset of even size $s \leqslant g$ such that $M_{\langle J \rangle}$ contains exotic Tate classes, then $J = I_0$ or $J = \overline I_0$.
	\end{lemma}
	\begin{proof}
		Let $J$ be as in the theorem, we have by \Cref{lemma_partition} that $J \cup \overline J = \{1, \dots, 2g\}$ and $\#J = g$. We need to understand $J$ with respect to the partition defined in the proof of \Cref{lemma_noncommutative}:
		$$\{1, \dots, 2g\} = I' \cup I'' \cup \overline {I'} \cup \overline{I''}. $$
		First, we will show that $(I'\cup \overline{I'})\cap J$ and $(I''\cup \overline{I''})\cap J$ are of cardinal $g/2$. Indeed, as $J \cup \overline J = \{1, \dots, 2g \}$ we get by intersecting with $I'$ that $(J \cap I') \cup (\overline J \cap I')$ is of cardinal $g/2$. Hence we get
		\begin{align*}
		g/2 &\leqslant \#(J \cap I') + \# (\overline J \cap {I'}) \\
		& \leqslant \#(J \cap I') + \# ( J \cap \overline {I'}) \\
		&= \# \left [(J \cap I')\cup ( J \cap \overline {I'}) \right ]\\
		&= \# \left [J \cap (I'\cup \overline {I'}) \right ]
		\end{align*}
		and the same inequality holds for $I''$ instead of $I'$. As $J = \left [J \cap (I'\cup \overline {I'}) \right ] \cup \left [J \cap (I''\cup \overline {I''}) \right ]$ is of cardinal $g$, we get that the previous inequalities are in fact \textit{equalities}. Moreover, if $S \in W$ is one of the set of the partition, then $S\cap J$ is either $\emptyset$ or $S$. Indeed if it is non-empty and not $S$, then we will construct many conjugates of $J$, which will contradict the second condition of \Cref{meav} for $M_{\langle J \rangle}$.\\
	To do so, consider $i,j \in S$ such that $i \not \in J$ and $j \in J$. Then by transitivity of the action of $G$ on $\{1, \dots, 2g\}$, there exists $\sigma$ such that $\sigma \cdot j = i$. Hence by \Cref{lemma_action2} we get that $\sigma S = S$. Hence we have $\sigma J \not = J$ because $i \in \sigma J$ but $i \not \in J$. Now if we prove that $\sigma J \not = \overline J$, then $J$ would have too much conjugates, so that $M_{\langle J \rangle}$ could not satisfy the second condition of \Cref{meav}. We will do so by proving that the sizes of $\sigma J \cap S$ and $\overline J \cap S$ have different parities, because $\#(\sigma J \cap S) + \# (\overline J \cap S)$ is an odd number:
	\begin{align*}
		\#(\sigma J \cap S) + \# (\overline J \cap S) &= \#(\sigma J \cap \sigma S) + \#(\overline J \cap \overline {\overline S}) \\
		&= \#(J \cap S) + \# (J \cap \overline S) \\
		&= \# \left [ (J \cap S ) \cup (J \cap \overline S) \right ] \\
		&= \# \left [ J \cap (S \cup \overline S) \right ]\\
		&= g/2.
	\end{align*}
	Hence $S \cap J$ is either $S$ or $\emptyset$ for each $S \in W$, so that for size reasons, there are two elements $S$ of $W$ such that the intersection is empty and two such that the intersection is $S$. These two elements can be: $I'$ and $I''$ in which case $J$ is $I_0$, or it can be $\overline {I'}$ and $\overline {I''}$ in which case $J$ is $\overline I_0$. They cannot be $I'$ and $\overline {I'}$ because then $J = J \cap (I' \cup \overline {I'})$ would not be of size $g/2$, and the same for $I''$. The last possibilities are $I'$ and $\overline {I''}$, then $J = I' \cup \overline {I''}$, which is impossible because the Tate classes in $M_{\langle J \rangle }$ would be Lefschetz classes: $\{\pi_i|\; i\in \overline{I''}\}$ is the set of complex conjugates of $\{\pi_1, \dots, \pi_{g'}\}$. Moreover, the same argument works for $J = I'' \cup \overline{I'}$.
	\end{proof}

\section{External tensor products and Tate classes} \label{relation-representation}
In this section, we introduce a tool to describe multiplicative relations between the Frobenius eigenvalues, some of which give rise to Tate classes by \Cref{rmk-rank-eigvalues}. Following the insight of Girstmair \cite{gir99}, we will use that the space of relations is a Galois module. More precisely, let $L$ be a CM field of degree $2g$ and $\pi \in L$ be a Weil $q$-number. We consider $\tilde L$ the Galois closure of $L$ and we enumerate the set $\Hom(L, \tilde L) = \{\sigma_1, \dots, \sigma_{2g}\}.$ For $1 \leqslant i \leqslant 2g$, let $\pi_i = \sigma_i(\pi)$, these are the conjugates of $\pi$ in $\tilde L$. 

\begin{definition}
Denote by $\bZ[\Hom(L, \tilde L)]$ or $\bZ[\sigma_1, \dots, \sigma_{2g}]$ the free abelian group on $\Hom(L, \tilde L)$.
Consider the map $\text{ev}_\pi$:
\begin{align*}
	\ev_\pi: \bZ [\sigma_1, \dots \sigma_{2g}] &\to \tilde L^{\times}/(p) \\
	a_1 [\sigma_1] + \dots + a_{2g} [\sigma_{2g}] &\mapsto \pi_1^{a_1} \dots {\pi_{2g}}^{a_{2g}} \pmod p .
\end{align*}
The space of multiplicative relations $K_\pi$ between the $\pi_i$'s is defined as the kernel of $\ev_\pi$.
\end{definition}
\begin{remark}
	Let $\Gamma = \Gal(\tilde L/\bQ)$, both $\bZ [\sigma_1, \dots \sigma_{2g}]$ and $\tilde L^{\times}/(p)$ are $\Gamma$-modules. Moreover, $\ev_\pi$ is $\Gamma$-equivariant, hence $K_\pi$ is a subrepresentation of the permutation representation $\bZ [\sigma_1, \dots \sigma_{2g}]$. This is the point of view of Girstmair in \cite{gir99}. Then the next input is to put some conditions on $L$, so that $\Gamma = \mu_2 \times G \times G'$ and that $\bQ[\sigma_1, \dots \sigma_{2g}]$ becomes an external tensor product of representations of $\mu_2$, $G$ and $G'$ (c.f. \Cref{def_boxtimes}). This will allow us to describe the possible $\Gamma$-submodules of $\bQ [\sigma_1, \dots \sigma_{2g}]$, hence the possible values for $K_\pi \otimes \bQ$.
\end{remark}
\begin{definition} \label{norm_embedding}
	Let $L$ be a number field, $\tilde L$ be a Galois closure of $L$. For $E$ be a subfield of $L$ and $\sigma \in \Hom(L, \tilde L)$, we denote by $N_{L/E}(\sigma)$ the following element of $\bZ[\Hom(L, \tilde L)]$:
	\[\sum_{\substack{\nu: L \to \tilde L \\ \nu_{|E}=\sigma_{|E}}} [\nu]. \]
	Moreover, this construction extends uniquely to a map
	\[ N_{L/E} : \bZ[\Hom(L, \tilde L)] \to \bZ[\Hom(L, \tilde L)].\]
\end{definition}
\begin{lemma} \label{lemma_norm}
	For $x \in L$, $E$ a subfield $L$ and $\sigma: L \to \tilde L$ any embedding, we have
	\[\sigma(N_{L/E}(x)) = \ev_{x}(N_{L/E}(\sigma)). \]
\end{lemma}
This lemma is a reformulation in this context of the fact that the norm of an element is the product of its conjugates.
\begin{definition} \label{def_boxtimes}
	Let $k$ be a field of characteristic $0$, $G$ and $G'$ be finite groups, $V$ a $k$-linear representation of $G$ and $V'$ a $k$-linear representation of $G'$. The external tensor product $V \boxtimes_k V'$ of $V$ and $V'$ is the $k$-vector space $V \otimes_k V'$ endowed with the following $k$-linear action of $G \times G'$:
	\[(g,g')\cdot (v \otimes v') := (gv) \otimes (g'v'). \]
\end{definition}
\begin{proposition} \label{irredtensor}
	Let $G$ and $G'$ be finite groups, $V$ a complex representation of $G$ and $V'$ a complex representation of $G'$. Assume that $V$ and $V'$ are irreducible finite dimensional representations of $G$ and $G'$. Then $V \boxtimes_{\bC} V'$ is an irreducible representation of $G \times G'$.
\end{proposition}
\Cref{irredtensor} is \cite[Theorem 10]{ser77}.
\begin{theorem} \label{prop_relation}
Let $Q$ be a quadratic imaginary field, $R$ and $R'$ be totally real number fields of degree $m$ and $g'$. Let $\tilde R$, $\tilde {R'}$ be Galois closures of $R$, $R'$ and
\begin{align*}
	G &= \Gal(\tilde R/\bQ), \\
	G' &= \Gal(\tilde{R'}/\bQ).
\end{align*}
Let $L$ be a compositum of $Q$, $R$ and $R'$, and $\pi \in L$ be a Weil $q$-number. Assume moreover that
\begin{itemize}
\item $\tilde R$ and $\tilde{R'}$ are linearly disjoint.
\item $N_{L/QR}(\pi) = q^{g'/2}$.
\item No powers of $N_{L/{QR'}}(\pi)$ are in $\bQ$.
\item $G$ acts $2$-transitively on $\Hom(R, \tilde R)$ (or $G = \{e\}$).
\item $\deg( R') \geqslant 2$ and the action of $G'$ on $\Hom(R', \tilde R')$ is $2$-transitive.
\end{itemize}
Then either for all $k$, $\pi^k \not \in Q R'$ and the rational space $K_\pi \otimes_{\bZ} \bQ$ of relations between the $\pi_i$'s is generated by 
\[\left \{N_{L/QR}(\sigma), [\sigma] +[\overline \sigma]  \; |\; \sigma \in \Hom(L, \tilde L) \right \}.\]
Or there is some $k$ such that $\pi^k \in Q R'$, and the rational space of relations is generated by 
\[\left \{N_{L/QR}(\sigma), [\sigma] +[\overline \sigma], \; [\mu \sigma]  - [\sigma] \; |\; \sigma \in \Hom(L, \tilde L), \mu \in \Gal(\tilde L/ Q\tilde R') \right \} .\]
\end{theorem}
\begin{remark}
	By \Cref{prop_relation} applied with $R=\bQ$ and \Cref{exoset_exomot}, we can recover \Cref{lemma_relations}.
\end{remark}
\begin{proof}[Proof of \Cref{prop_relation}]
	It follows from $N_{L/QR}(\pi) = q^{g'/2}$ and $\pi \overline \pi = q$ that the given generators are indeed elements of $K_\pi \otimes \bQ$ in the first case. In the second case, we have in addition that $\pi^k \in Q R'$. Hence $\mu(\sigma(\pi^k)) = \sigma(\pi^k)$ for all $\sigma \in \Hom(L, \tilde L)$ and $\mu \in \Gal(\tilde L/ Q \tilde R')$, so that $k \left ( [\mu \sigma] - [\sigma] \right ) \in K_\pi$. We are left to show that these sets indeed generate $K_\pi \otimes \bQ$.

	We have that $K_\pi\otimes \bQ$ is a subrepresentation of the permutation representation
        \[P = \bQ[\Hom(L, \tilde L)].\]
        Hence to understand $K_\pi \otimes \bQ$, it suffices to understand the decomposition of $P$ as a sum of simple factors, and to find which factors appear in $K_\pi$. By \Cref{lemma-permutation-box}, $P$ is an external tensor product, and by \Cref{irredtensor}, external tensor product of irreducible representations are again irreducible. Hence understand irreducible subrepresentations of $P$, it suffices to study irreducible subrepresentations of each factor independently. This study is done in \Cref{decomp-perm}, \Cref{decomp-pm}, \Cref{decomp-sm} and \Cref{lemma-case-disj}. Here are the conclusion:
        \begin{itemize}
        \item $P$ decomposes as $P = P_+ \oplus S_+ \oplus T_+ \oplus T_-$,
        \item $P_+$ is generated by $\big \{[\sigma] +[\overline \sigma]  \; |\; \sigma \in \Hom(L, \tilde L) \big \}.$
        \item According to \Cref{lemma-case-disj}, $P_+ \oplus S_+$ is generated by
          \[\left  \{N_{L/QR}(\sigma), [\sigma] + [\overline \sigma] \; |\; \sigma \in \Hom(L, \tilde L) \right \}, \]
          and $P_+ \oplus S_+ \oplus T_-$ is generated by
          \[\left \{N_{L/QR}(\sigma), [\sigma] +[\overline \sigma], \; [\mu \sigma]  - [\sigma] \; |\; \sigma \in \Hom(L, \tilde L), \mu \in \Gal(\tilde L/ Q\tilde R') \right \} .\]
        \item $P_+ \oplus S_+ \subset K_\pi \otimes \bQ$ (\Cref{decomp-perm} and \Cref{decomp-pm}),
        \item $\big ( K_\pi \otimes \bQ \big ) \cap T_+ = 0$ (\Cref{decomp-sm}),
        \item By \Cref{lemma-case-disj}, either no power of $\pi$ is in $QR'$ and $\big ( K_\pi \otimes \bQ \big ) \cap T_- = 0$,

          or some power of $\pi$ is in $QR'$ and $T_- \subset K_\pi \otimes \bQ$.
        \end{itemize}
        
        To conclude, either no power of $\pi$ is in $QR'$ and $K_\pi \otimes \bQ$ is the subrepresentation of $P$ satisfying
        \[ K_\pi\otimes \bQ = P_+ \oplus Q_+, \]
        which is generated by
        \[\left \{N_{L/QR}(\sigma), [\sigma] +[\overline \sigma]  \; |\; \sigma \in \Hom(L, \tilde L) \right \}.\]
        Or some power of $\pi$ is in $QR'$, and 
        \[ K_\pi \otimes \bQ = P_+ \oplus Q_+ \oplus T_-, \]
        which is generated by 
        \[\left \{N_{L/QR}(\sigma), [\sigma] +[\overline \sigma], \; [\mu \sigma]  - [\sigma] \; |\; \sigma \in \Hom(L, \tilde L), \mu \in \Gal(\tilde L/ Q\tilde R') \right \} .\]
      \end{proof}
	
	\begin{lemma} \label{lemma-permutation-box}
          Let $Q$, $R$, $R'$ be fields as in \Cref{prop_relation}, and $P = \bQ[\Hom(L, \tilde L)]$ be the permutation representation of $\Gal(\tilde L/ \bQ)$. Then
          \begin{align*}
            \Gal(\tilde L/ \bQ) &\cong \Gal(Q/\bQ) \times \Gal(\tilde R/ \bQ) \times \Gal(\tilde R'/\bQ), \\
            P &\cong \bQ[\Hom(Q, Q)] \boxtimes \bQ[\Hom(R, \tilde R)] \boxtimes \bQ[\Hom(R', \tilde R')].
          \end{align*}
	\end{lemma}
	\begin{proof}
	As $Q$ is imaginary and by the disjointness assumptions, we have 
	\begin{align*}
		L &= Q \otimes R \otimes R', \\
		\tilde L &= Q \otimes \tilde R \otimes \tilde R'.
	\end{align*}
	Tensor products are colimit in the category of $\bQ$-algebras, hence
	\begin{align*}
		\Hom(L, \tilde L) &=  \Hom(Q,Q) \times \Hom(R, \tilde R) \times \Hom(R', \tilde{R}'), \\
		\Gal(\tilde L/\bQ) &= \Gal(Q/\bQ) \times \Gal(\tilde R/\bQ) \times \Gal(\tilde R'/ \bQ).
	\end{align*}
	The above descriptions are compatible with the actions, that is if $(\varepsilon, \sigma, \sigma') \in \Gal(\tilde L/\bQ)$ and $(a,b,c) \in \Hom(L, \tilde L)$, then 
	\[(\varepsilon, \sigma, \sigma') \cdot (a,b,c) = (\varepsilon a, \sigma b, \sigma' c). \]
	Hence we have that
	\[ P = \bQ[\Hom(Q, Q)] \boxtimes \bQ[\Hom(R, \tilde R)] \boxtimes \bQ[\Hom(R', \tilde R')]. \]
	\end{proof}
	
	\begin{lemma} \label{decomp-perm}
		Let $L$, $Q$, $R$, $R'$ be fields as in \Cref{prop_relation} and $P$ be the permutation representation of $\Gal(\tilde L/\bQ)$ acting on $\Hom(L, \tilde L)$. Then $P$ decomposes as $P = P_+ \oplus P_-$ where $P_+$ (resp. $P_-$) is the isotypical component of type $\chi$ for the action of $\Gal(Q/\bQ)\cong \mu_2$ on $P$ where $\chi$ is the trivial representation (resp. the sign representation). Moreover $P_+ \subset K_\pi \otimes \bQ$.
	\end{lemma}
	\begin{proof}
	First, $\bQ[\Hom(Q,Q)] = \bUn \oplus \varepsilon$ where $\bUn$ is the trivial representation and $\varepsilon$ is the sign representation of $\Gal(Q/\bQ)=\mu_2$. Hence it follows from \Cref{lemma-permutation-box} that
	\begin{align*}
		P =  \bUn \boxtimes \bQ[\Hom(R, \tilde R)] &\boxtimes \bQ[\Hom(R', \tilde R')] \qquad \quad \big \}\; P_+ \\
		&\bigoplus \\
		 \varepsilon \boxtimes \bQ[\Hom(R, \tilde R)] &\boxtimes \bQ[\Hom(R', \tilde R')] \qquad \quad \big \}\; P_-
	\end{align*}
	Now, the relation $\pi \overline \pi = q$ tells us that for any $\sigma \in \Hom(L, \tilde L)$, we have $[\sigma] + [\overline \sigma] \in K_\pi$. Moreover, $\bUn \subset \bQ[\Hom(Q,Q)]$ is the line generated by $[\id_Q] + [\tau]$ where $\tau$ is the complex conjugation of $Q$. Hence an element of $P_+$ can be written as a linear combination of $[\sigma] + [\overline \sigma]$, which are in $K_\pi$. We conclude that $P_+ \subset K_\pi \otimes \bQ$.
	\end{proof}
	
	\begin{lemma} \label{decomp-pm}
          Let $L$, $Q$, $R$, $R'$ be fields as in \Cref{prop_relation} and $P = P_+ \oplus P_-$ as in \Cref{decomp-perm}. Then $P_-$ decomposes as $P_- = S_+ \oplus S_-$ where $S_+$ (resp. $S_-$) is the isotypical component of type $\chi$ for the action of $\Gal(\tilde R'/ \bQ)$, where $\chi$ is the trivial representation (resp. the standard representation).
          Moreover $S_+ \subset K_\pi \otimes \bQ$.
	\end{lemma}
	\begin{proof}
          Let us decompose $\bQ[\Hom(R', \tilde R')] = \bUn \oplus I'$, where 
          \[\bUn \subset \bQ[\Hom(R', \tilde R')]\]
          is the line generated by $\sum_{\mu \in \Hom(R', \tilde R')} [\mu]$ and $I'$ is the orthogonal to $\bUn$ for the euclidean scalar product. As $G'$ acts $2$-transitively on $\Hom(R', \tilde R')$, we have that $I'\otimes_{\bQ} \bC$ is an irreducible representation by \cite[Exercise 2.6]{ser77}, that we will call the standard representation of $\Gal(\tilde R'/ \bQ)$.
          From
          \[P_{-} = \varepsilon \boxtimes \bQ[\Hom(R, \tilde R)] \boxtimes \bQ[\Hom(R', \tilde R')],\]
          it follows that
	\begin{align*}
		P_{-} = \varepsilon \boxtimes \bQ[\Hom(&R, \tilde R)] \boxtimes \bUn \qquad \quad \big \}\; S_+ \\
				& \bigoplus \\
			\varepsilon \boxtimes \bQ[\Hom(&R, \tilde R)] \boxtimes I' \qquad \quad \big \}\; S_-
	\end{align*}
	From the fact that $N_{L/QR}(\pi) = q^{g'}$, we deduce that $N_{L/QR}(\pi \cdot \overline {\pi}^{-1}) = 1$. Then \Cref{lemma_norm} tells us that for all $\sigma \in \Hom(L, \tilde L)$ we have,
	\[ N_{L/QR}(\sigma) - N_{L/QR}(\overline \sigma) \in K_\pi. \]
	Elements of $S_+$ are linear combinations of these differences of norms because $\varepsilon$ is generated by $[\id_Q] - [\tau]$ and $\bUn \subset \bQ[\Hom(R', \tilde R')]$ is generated by $N_{R'/\bQ}(\mu)$ for any $\mu \in \Hom(R', \tilde R')$. Hence $S_+ \subset K_\pi\otimes \bQ$.
	\end{proof}

        \begin{lemma} \label{decomp-sm}
          Let $L$, $Q$, $R$, $R'$ be fields as in \Cref{prop_relation} and $P_- = S_+ \oplus S_-$ as is \Cref{decomp-pm}. Then $S_-$ decomposes as $S_- = T_+ \oplus T_-$ where $T_+$ (resp. $T_-$) is the isotypical component of type $\chi$ for the action of $\Gal(\tilde R/ \bQ)$, where $\chi$ is the trivial representation (resp. the standard representation). Moreover $T_+$ is irreducible and $T_-$ is either $\{0\}$ or an irreducible representation of $\Gal(\tilde L/ \bQ)$, and
          \[ \big ( K_\pi \otimes \bQ \big ) \cap T_+ = \{0\}. \]
        \end{lemma}

        \begin{proof}
          We decompose $\bQ[\Hom(R, \tilde R)] = \bUn \oplus I$ as a representation of $\Gal(\tilde R/ R)$, where $\bUn$ denotes the line generated by 
          \[ \sum_{\mu \in \Hom(R, \tilde R)} [\mu]\]
          and where $I$ is the orthogonal of $\bUn$ with respect to the euclidean scalar product. Then,
          \[S_-= \varepsilon \boxtimes \bQ[\Hom(R, \tilde R)] \boxtimes I'\]
          decomposes as
          \[ S_- = \left ( \varepsilon \boxtimes \bUn \boxtimes I' \right ) \oplus \left ( \varepsilon \boxtimes I \boxtimes I' \right ).\]
          in which the first summand will be denoted $T_+$, and the second summand $T_-$.

          By \Cref{irredtensor}, we have that $T_+ \otimes \bC$ is irreducible, hence also $T_+$.
          
          We also have that $I$ is generated by the vectors $[\mu] - [\mu']$ with $\mu, \mu' \in \Hom(R, \tilde R)$. Moreover either $G$ is trivial and $I = 0$ as well as $T_-$, or $G$ acts $2$-transitively on $\Hom(R, \tilde R)$ and $I \otimes_{\bQ} \bC$ is an irreducible representation of $G$ by \cite[Exercise 2.6]{ser77}, that we will call the standard representation.

          In case $T_- \not = 0$, then $T_- \otimes \bC$ is irreducible by \Cref{irredtensor} and so is $T_-$.
          
          Assume that $T_+$ is in $K_\pi \otimes \bQ$, it follows from \Cref{decomp-perm} and \Cref{decomp-pm} that
          \begin{equation} \label{eq:inclusion}
            P_+ \oplus S_+ \oplus T_+ \subset K_\pi \otimes \bQ.
          \end{equation}
          We also have
          \[\bQ[\Hom(Q, Q)] \boxtimes \bUn \boxtimes \bQ[\Hom(R', \tilde R')] \subset P_+ \oplus S_+ \oplus T_+. \]
          Hence $P_+ \oplus S_+ \oplus T_+$ contains $N_{L/QR'}(\sigma)$ for all $\sigma \in \Hom(L, \tilde L)$, and it follows from \eqref{eq:inclusion} that $N_{L/QR'}(\sigma_1)$ is in $K_\pi \otimes \bQ$. Moreover, \Cref{lemma_norm} tells us that some power of $N_{L/QR'}(\pi)$ is in $\bQ$, and this contradicts the hypothesis of \Cref{prop_relation}. As $T_+$ is irreducible, we deduce that
          \[ \big ( K_\pi \otimes \bQ \big )\cap T_+= \{0\}. \]
        \end{proof}
        \begin{lemma} \label{lemma-case-disj}
          Let $L$, $Q$, $R$, $R'$ be fields as in \Cref{prop_relation} and $T_-$ the representation of \Cref{decomp-sm}.
          Then either some power of $\pi$ is in $QR'$ and
          \[T_- \subset K_\pi \otimes \bQ ,\]
          or no power of $\pi$ is in $QR'$ and
          \[\big ( K_\pi \otimes \bQ \big ) \cap T_- = \{0\}.\]
          Moreover $P_+ \oplus S_+$ is generated by
          \[\left \{N_{L/QR}(\sigma), [\sigma] +[\overline \sigma]  \; |\; \sigma \in \Hom(L, \tilde L) \right \},\]
          and $P_+ \oplus S_+ \oplus T_-$ is generated by
          \[\left \{N_{L/QR}(\sigma), [\sigma] +[\overline \sigma], \; [\mu \sigma]  - [\sigma] \; |\; \sigma \in \Hom(L, \tilde L), \mu \in \Gal(\tilde L/ Q\tilde R') \right \} .\]
        \end{lemma}
        \begin{proof}
          Either $T_-$ is $\{0\}$, in which case $R=\bQ$ and the result is clear. Otherwise, $T_-$ is irreducible by \Cref{decomp-sm}. By \Cref{decomp-perm}, \Cref{decomp-pm} and \Cref{decomp-sm}, we have that $T_- \subset K_\pi \otimes \bQ$ if and only if 
          \[  P_+ \oplus S_+ \oplus T_- \subset K_\pi \otimes \bQ .\]
          Moreover,
          \begin{equation} \label{eq:inclusion2}
            T_- \subset \bQ[\Hom(Q, Q)] \boxtimes I \boxtimes \bQ[\Hom(R', \tilde R')] \subset P_+ \oplus S_+ \oplus T_-.
          \end{equation}
          Hence $T_- \subset K_\pi\otimes \bQ$ if and only if
          \[\bQ[\Hom(Q, Q)] \boxtimes I \boxtimes \bQ[\Hom(R', \tilde R')] \subset K_\pi \otimes \bQ.\]
          As $I$ is generated by $[\mu] - [\mu']$ for $\mu, \mu' \in \Hom(R, \tilde R)$, we have $T_- \subset K_\pi \otimes \bQ$ if and only if 
          \[ [\gamma \sigma] - [\sigma] \in K_\pi \otimes \bQ \]
          for all $\sigma \in \Hom(L, \tilde L)$ and $\gamma \in \Gal(\tilde L/ Q \tilde R')$. This happens if and only if there exists $k$ such that $\gamma \cdot \pi^k = \pi^k$ for all $\gamma \in \Gal(\tilde L/Q \tilde R')$. Hence $T_- \subset K_\pi \otimes \bQ$ if and only if some power of $\pi$ is in $Q R'$ (because $QR' = L \cap (Q\tilde R')$).

          From
          \[ S_+ \subset \bQ[\Hom(Q, Q)] \boxtimes \bQ[\Hom(R, \tilde R)] \boxtimes \bUn \subset P_+ \oplus S_+,\]
          we deduce that $P_+ \oplus S_+$ is generated by
          \begin{align*}
            &\bQ[\Hom(Q, Q)] \boxtimes \bQ[\Hom(R, \tilde R)] \boxtimes \bUn \text{ and } \\
            &\bUn \boxtimes \bQ[\Hom(R, \tilde R)] \boxtimes \bQ[\Hom(R' \tilde R')],
          \end{align*}
          which are generated by
          \[\left \{N_{L/QR}(\sigma), [\sigma] +[\overline \sigma]  \; |\; \sigma \in \Hom(L, \tilde L) \right \}.\]
          Moreover by \eqref{eq:inclusion2}, it follows that $P_+ \oplus S_+ \oplus T_-$ is generated by
          \[\left \{N_{L/QR}(\sigma), [\sigma] +[\overline \sigma], \; [\mu \sigma]  - [\sigma] \; |\; \sigma \in \Hom(L, \tilde L), \mu \in \Gal(\tilde L/ Q\tilde R') \right \} .\]
        \end{proof}
        
\section{More examples of mildly exotic varieties} \label{proof_nc_com}
		
	The goal of this section if to prove \Cref{nc_case} and \Cref{com_case}. Fix a prime number $p$. The proofs revolve around choices of CM-fields $L$ and of CM-types $\Phi$ for $L$. The main difference between the proofs of \Cref{nc_case} and \Cref{com_case} is that fields extensions are ramified at $p$ in \Cref{ramified_diagram} and unramified in \Cref{nr_diagram}. 
		
		Let $g' \geqslant 3$ be an odd number, and $p$ be a prime number. Let $\bQ_{p^2} = \Qp[\sqrt u]$ denote the unramified extension of $\Qp$ of degree $2$ where $u$ is an integer which is not a square in $\Qp$, and let $W$ be the unique unramified extension of $\Qp$ of degree $g'-1$. 
    \begin{lemma} \label{ramified_diagram}
    There exists infinitely many fields $F$ for which there is a CM-field $L$ whose only subfields are described by this diagram: \\
		\begin{minipage}{\textwidth}
		\centering
		\begin{tikzpicture}

    \node (Q1) at (0,0) {$\mathbb{Q}$};
    \node (Q2) at (-5,2) {$Q$};
    \node (Q3) at (-3,2) {$Q'$};
    \node (Q4) at (-1,2) {$B_{\bR}$};
    \node (Q5) at (3,2) {$R_0$};
	\node (Q6) at (-3, 4) {$B$};
	\node (Q7) at (1, 4) {$F$};    
	\node (Q8) at (3, 4) {$F'$};    
	\node (Q9) at (5, 4) {$R$};    
    
    \node (Q10) at (0,6) {$L$};

   	\draw (Q1)--(Q2) node [pos=0.5, left,inner sep=0.40cm] {2};
    \draw (Q1)--(Q3) node [pos=0.5, left,inner sep=0.25cm] {2};
    \draw (Q1)--(Q4) node [pos=0.5, left,inner sep=0.1cm] {2};
    \draw (Q1)--(Q5) node [pos=0.5, right,inner sep=0.25cm] {$g \prime$};
    \draw (Q2)--(Q6);
    \draw (Q3)--(Q6);
    \draw (Q4)--(Q6);
    \draw (Q6)--(Q10) node [pos=0.5, right,inner sep=0.25cm] {$g \prime$};
    \draw (Q5)--(Q7) node [pos=0.5, right,inner sep=0.25cm] {2};
    \draw (Q5)--(Q8) node [pos=0.5, right,inner sep=0.25cm] {2};
    \draw (Q5)--(Q9) node [pos=0.5, right,inner sep=0.40cm] {2};
    \draw (Q7)--(Q10);
    \draw (Q8)--(Q10);
    \draw (Q9)--(Q10);
    \draw[dotted] (Q2)--(Q7);
    \draw[dotted] (Q3)--(Q8);
    \draw[dotted] (Q4)--(Q9);

    \end{tikzpicture}
    \end{minipage}
    with $Q$ and $Q'$ quadratic imaginary fields, $B_{\bR}$, $R_0$ totally real such that the Galois closure $\tilde R_0$  of $R_0$ is disjoint from $B_{\bR}$, with $\Gal(\tilde {R_0}/\bQ)=S_{g'}$, and such that after $-\otimes \Qp$ we get the following diagram: \\
    \begin{minipage}{\textwidth}
    \begin{center}
    		\begin{tikzpicture}[scale=0.9]
    \node (Q1) at (0,0) {$\Qp.$};
    \node (Q2) at (-5,2) {$\bQ_{p^2}$};
    \node (Q3) at (-3,2) {$\Qp[\sqrt p]$};
    \node (Q4) at (-1,2) {$\Qp[\sqrt {up}]$};
    \node (Q5) at (4,2) {$W \times \bQ_p$};
	\node (Q6) at (-3, 4) {$\bQ_{p^2}[\sqrt p]$};
	\node (Q7) at (1, 4) {$W \times W \times \bQ_{p^2}$};    
	\node (Q8) at (4, 4) {$W[\sqrt p] \times \Qp[\sqrt p]$};    
	\node (Q9) at (8, 4) {$W[\sqrt p] \times \Qp[\sqrt {up}]$};    
    
    \node (Q10) at (0,6) {$W[\sqrt p] \times W[\sqrt p] \times \bQ_{p^2}[\sqrt p]$};

   	\draw (Q1)--(Q2);
    \draw (Q1)--(Q3);
    \draw (Q1)--(Q4);
    \draw (Q1)--(Q5);
    \draw (Q2)--(Q6);
    \draw (Q3)--(Q6);
    \draw (Q4)--(Q6);
    \draw (Q6)--(Q10);
    \draw (Q5)--(Q7);
    \draw (Q5)--(Q8);
    \draw (Q5)--(Q9);
    \draw (Q7)--(Q10);
    \draw (Q8)--(Q10);
    \draw (Q9)--(Q10);
    \draw[dotted] (Q2)--(Q7);
    \draw[dotted] (Q3)--(Q8);
    \draw[dotted] (Q4)--(Q9);

    \end{tikzpicture}
    \end{center}
    \end{minipage}
    \end{lemma}
    \begin{proof}
    		Let $T$ and $C$ be distinct prime numbers, different from $p$, and consider $Q, Q'$ quadratic imaginary fields with $p$ inert in $Q$ and $Q' \otimes_{\bQ} \Qp = \Qp[\sqrt p]$. There are infinitely many choices for $Q$. Let $B = Q \cdot Q'$, and $B_{\bR}$ be the real quadratic subfield of $B$. We have that $Q \otimes \Qp = \Qp[\sqrt u]$ and $Q' \otimes \Qp = \Qp[\sqrt p]$, hence $B_{\bR}\otimes \Qp = \Qp[\sqrt {up}]$. By \Cref{lemma_real_fields}, there are infinitely many totally real fields $R_0$ of degree $g'$, such that the Galois closure $\tilde {R_0}$ is disjoint from $B_{\bR}$, such that $R_0 \otimes \Qp = W \times \Qp$ and with $\Gal(\tilde{R_0}/\bQ) = S_{g'}$. Then $L = Q \cdot Q' \cdot R_0$ will satisfy the assumptions of the lemma. Indeed, $\tilde L = Q \otimes (B_{\bR} \otimes \tilde{R_0})$, so that the Galois group $\Gamma$ can be written as a product
	 \[\Gal(\tilde L/\bQ) = \mu_2 \times \mu_2 \times S_{g'}. \]
	 The subgroup of $\Gamma$ fixing $L$ is $H_L= \{1\} \times \{1\} \times \Fix(1)$, hence the subgroups of $\Gamma$ containing $H_L$ correspond via Galois theory to subfields of $L$. A group theoretic computation gives us the desired diagram. Moreover we see that subfields of $L$ are obtained by taking compositum of $Q, Q', B_{\bR}$ and $R_0$. Hence the diagram $\otimes \Qp$ follows from computations of tensor products of finite $\Qp$-algebras.
    	\end{proof}
    
    \begin{lemma} \label{nc_abelianscheme}
    		For any field $L$ satisfying \Cref{ramified_diagram}, we have three prime ideals $\frk p$, $\overline {\frk p}$ and $\frk q$ above $p$ in $\mathcal O_L$, with valuations $v, \overline{v}$ and $w$, and the complex conjugation of $L$ swaps $\frk p$ and $\overline {\frk p}$ and fixes $\frk q$. There is an abelian scheme $\mathcal A$ over the ring of integers of a $p$-adic field $K$ with complex multiplication by $L$, and $\pi \in L$ acting as the Frobenius on the special fibre of $\mathcal A$, with valuations:
    \begin{equation*}
    \left\{
    \begin{array}{l}
    		v(\pi) = \frac{1}{2g'-2} \\
    		\overline{v}(\pi) = \frac{2g'-3}{2g'-2} \\
    		w(\pi) = \frac{1}{2}.
    	\end{array} \right.
    \end{equation*}
    
    	\end{lemma}    		
    	\begin{proof}
    	By \cite[\S 4]{tat71}, we can choose a CM-type $\Phi$ for $L$ such that:
    \begin{equation*}
    \left\{
    \begin{array}{l}
    		\#\Phi_{v} = 1 \\
    		\#\Phi_{\overline{v}} = 2g'-3 \\
    		\#\Phi_{w} = 2.
    	\end{array} \right.
    \end{equation*}
    By \Cref{thm_bourbaki}, we have a $p$-adic field $K$ and an abelian scheme $\mathcal A$ over the ring of integers $\mathcal O_K$, with complex multiplication by $L$ of type $\Phi$. Hence by \Cref{S-T_formula}, we have $\pi \in L$ acting as the Frobenius on the special fibre of $\mathcal A$, with valuations prescribed as in the lemma, given in terms of the CM-type $\Phi$.
    \end{proof}
	\begin{lemma} \label{nc_avar}
		Let $F$ and $L$ be as in \Cref{ramified_diagram}, then we can choose $\mathcal A$ in \Cref{nc_abelianscheme} such that its special fiber $A$ is a geometrically simple abelian variety with non-commutative algebra of endomorphisms. More precisely, $\End^0(A)$ will be a quarternion algebra over $F$ split by $L$, and $F = \bQ(\pi)$ where $\pi$ is the Frobenius endomorphism of $A$.
	\end{lemma}
    \begin{proof}
	    Let $\mathcal A$ be any abelian scheme satisfying \Cref{nc_abelianscheme}, and $\pi \in L$ corresponding to the Frobenius endomorphism of the special fibre. By \Cref{easy_case}, we have $k$ such that $\pi^k \in F$, we will assume $k = 1$ (we can take a degree $k$ extension of the residual field). Because $v(\pi) \not = \overline v(\pi)$ we have that $\pi$ is not in $B$ nor $F'$ or $R$. Hence $F = \bQ(\pi)$ and $[L:F] = 2$. We know by \cite[Theorem 3]{tat66} that the special fibre $A$ of $\mathcal A$ is isogenous to a power of a simple variety $B$. We want to prove it is in fact simple, to do so we compute the invariants of $\End^0(B)$ as a $F$-central algebra. Let $v'$ be the restriction of $v$ to $F$, we have 
	    \[\text{inv}_{v'}(E) = (g'-1) \times v(\pi) = \frac{1}{2}\]
	   which is not an integer. Hence \cite[Theorem 1]{tat71} allows us to compute $\dim(B) = [F:\bQ]$, which is also the dimension of $A$, hence $A = B$ is simple. Moreover, from $[L:\bQ(\pi)]=2$ and  \cite[Theorem 1]{tat71}, we deduce that $[\End^0(A): \bQ(\pi)]=4$.
    \end{proof}
    \begin{proof}[Proof of \Cref{nc_case}]
    Let $A/\bF_q$ be an abelian variety satisfying \Cref{nc_avar}. Following \Cref{rmk-rank-eigvalues}, to construct submotives of $A$, we have to exhibit non-trivial relations between the conjugates of the Weil $q$-number $\pi$ associated to $A$. The norm 
    \[N_{L/B}(\pi) \in B\]
is a Weil $q$-number, and $B/\bQ$ satisfies the hypothesis of \Cref{easy_case}. Hence there exists an integer $k > 0$ such that $N_{L/B}(\pi^k)$ is a power of $q$. Moreover, $\pi^k$ is the Frobenius endomorphism of the extension of scalars of $A$ to $\bF_{q^k}$, and there is an extension $K'/K$ of $p$-adic fields such that $\mathcal A':= \mathcal A \times_{\mathcal O_K} \mathcal O_{K'}$ has $A\times_{\bF_q} \bF_{q^k}$ as its special fiber. If we replace $\mathcal A$ by $\mathcal A'$, we are still in the setup of \Cref{nc_abelianscheme}. Hence we can assume that 
\[N_{L/B}(\pi) = q^{g'/2}\]
and by transivity of norms we get
\[N_{L/Q}(\pi) = q^{g'}.\]
This last relation gives rise to an exotic submotive $M$ in $\frk h(A)$. We now want to prove that $M$ is the only exotic submotive, hence we have to study multiplicative relations between the conjugates of $\pi$. Now apply \Cref{prop_relation} with $Q:=Q$, $R:= B_{\bR}$ and $R':= R_0$. Indeed, we have that 
	\begin{itemize}
		\item $B_{\bR}$ and $\tilde R_0$ are linearly disjoint
		\item $N_{L/B}(\pi) = q^{g'/2}$
		\item No powers of $N_{L/F}(\pi) = \pi^2$ are in $\bQ$. Indeed, $v(\pi) \not = \overline v(\pi)$ hence no power of $\pi$ is stable by complex conjugation.
		\item $\mu_2$ acts $2$-transitively on $\Hom(B_{\bR}, B_{\bR})$
		\item $g' \geqslant 2$ and $S_{g'}$ acts $2$-transitively on $\Hom(R_0, \tilde R_0) = \{1, \dots, g'\}$
	\end{itemize}
We are in the case where $\pi \in Q R_0 = F$, hence we get that the relations between the conjugates of $\pi$ are generated by $N_{L/B}(\pi) = q^{g'/2}$, $\pi  \overline \pi = q$ and $\sigma(\pi) = \sigma'(\pi)$ for $\sigma, \sigma': L \to \tilde L$ such that $\sigma_{|F} = \sigma'_{|F}$. But $N_{L/B}(\pi)$ is a product of $g'$ conjugates of $\pi$, and $g'$ is odd, hence the associated motive is of odd weight, and it cannot contain Tate classes. As we know a set of generators for the relations between the conjugates of $\pi$, and by \Cref{exoset_exomot}, we have that the only relations contributing to exotic Tate classes are $N_{L/Q}(\pi) = q^{g'}$ and $N_{L/Q}(\overline \pi) = q^{g'}$. These two relations are conjugate to each other, hence the associated motive $M$ is of rank $2$ and $A$ is mildly exotic.
	\end{proof}
	\begin{remark}
		In the previous proof, one could ask if the relation $N_{L/Q'}(\pi) = q^{g/2}$ gives rise to another interesting submotive of $\frk h(\mathcal A)$. In fact, it gives rise to a motive which contains only Lefschetz classes, in other words it is non-exotic. This is because there is redundancy in the eigenvalues of the Frobenius acting on the special fibre $A$ of $\mathcal A$, coming from the non-commutativity of $\End^0(A)$. \\
		More explicitely for some numbering of the conjugates of $\pi$:
		\[N_{L/Q}(\pi) = \pi_1^2 \times \dots \times \pi_{g/2}^2 \]
		whereas
		\[N_{L/Q'}(\pi) = \pi_1 \overline{\pi_1} \times \dots \times \pi_{g/2} \overline{\pi_{g/2}}.\]
	\end{remark}
	
	Let $g' \geqslant 3$ be an odd number and $p$ be a prime number. Now we will prove \Cref{com_case}.
	\begin{lemma} \label{nr_diagram}
	 Recall that $W$ is the unique unramified extension of $\Qp$ of degree $g'-1$. There exists infinitely many CM-fields $L$ such that the subfields of $L$ are described by this diagram:\\
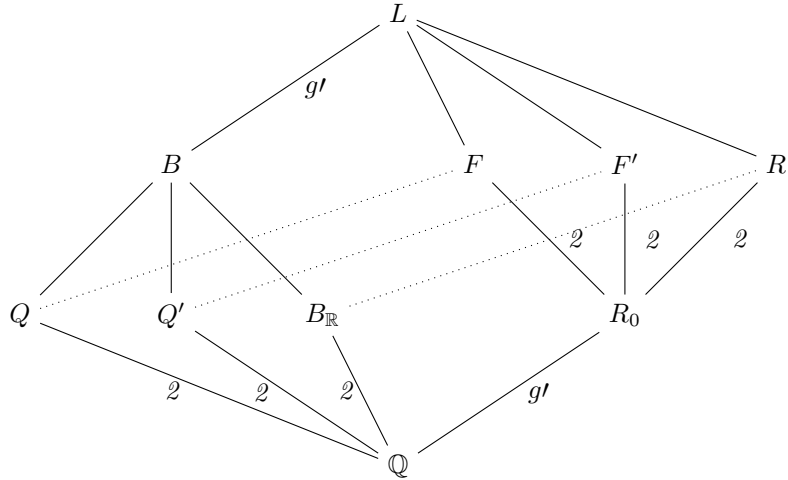
	\begin{minipage}{\textwidth}
	\centering
    	\begin{tikzpicture}[scale=1]
    \node (Q1) at (0,0) {$\mathbb{Q}$};
    \node (Q2) at (-5,2) {$Q$};
    \node (Q3) at (-3,2) {$Q'$};
    \node (Q4) at (-1,2) {$B_{\bR}$};
    \node (Q5) at (3,2) {$R_0$};
	\node (Q6) at (-3, 4) {$B$};
	\node (Q7) at (1, 4) {$F$};    
	\node (Q8) at (3, 4) {$F'$};    
	\node (Q9) at (5, 4) {$R$};    
    
    \node (Q10) at (0,6) {$L$};

   	\draw (Q1)--(Q2) node [pos=0.5, left,inner sep=0.40cm] {2};
    \draw (Q1)--(Q3) node [pos=0.5, left,inner sep=0.25cm] {2};
    \draw (Q1)--(Q4) node [pos=0.5, left,inner sep=0.1cm] {2};
    \draw (Q1)--(Q5) node [pos=0.5, right,inner sep=0.25cm] {$g \prime$};
    \draw (Q2)--(Q6);
    \draw (Q3)--(Q6);
    \draw (Q4)--(Q6);
    \draw (Q6)--(Q10) node [pos=0.5, right,inner sep=0.25cm] {$g \prime$};
    \draw (Q5)--(Q7) node [pos=0.5, right,inner sep=0.25cm] {2};
    \draw (Q5)--(Q8) node [pos=0.5, right,inner sep=0.25cm] {2};
    \draw (Q5)--(Q9) node [pos=0.5, right,inner sep=0.40cm] {2};
    \draw (Q7)--(Q10);
    \draw (Q8)--(Q10);
    \draw (Q9)--(Q10);
    \draw[dotted] (Q2)--(Q7);
    \draw[dotted] (Q3)--(Q8);
    \draw[dotted] (Q4)--(Q9);

    \end{tikzpicture}
    \end{minipage}
    with $Q$ and $Q'$ quadratic imaginary fields, $B_{\bR}$, $R_0$ totally real, such that the Galois closure $\tilde R_0$ of $R_0$ is disjoint from $B_{\bR}$, with $\Gal(\tilde {R_0}/\bQ)=S_{g'}$ and such that after $-\otimes \Qp$ we get the following diagram: \\
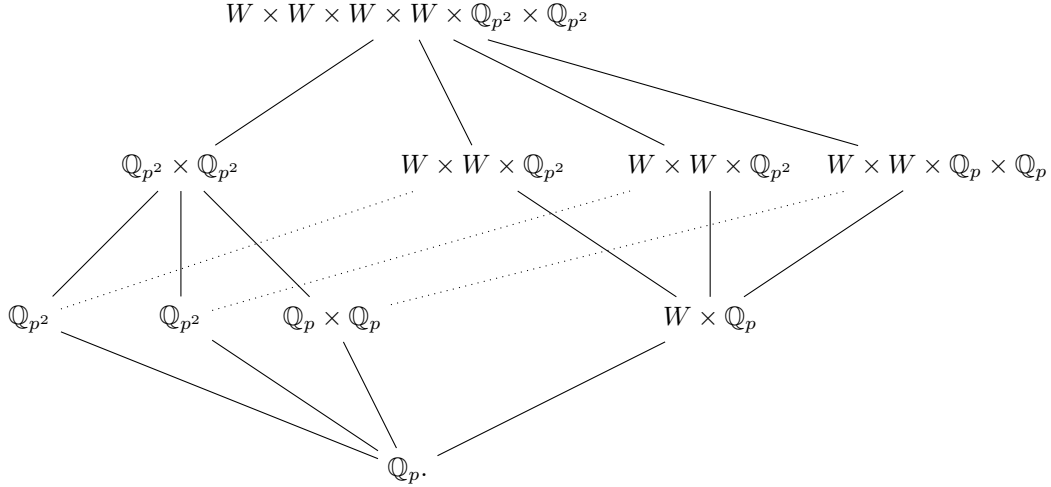
    	\begin{minipage}{\textwidth}
    	\centering
    	\begin{tikzpicture}[scale=1]
    \node (Q1) at (0,0) {$\Qp.$};
    \node (Q2) at (-5,2) {$\bQ_{p^2}$};
    \node (Q3) at (-3,2) {$\bQ_{p^2}$};
    \node (Q4) at (-1,2) {$\Qp \times \Qp$};
    \node (Q5) at (4,2) {$W \times \bQ_p$};
	\node (Q6) at (-3, 4) {$\bQ_{p^2} \times \bQ_{p^2}$};
	\node (Q7) at (1, 4) {$W \times W \times \bQ_{p^2}$};    
	\node (Q8) at (4, 4) {$W \times W \times \bQ_{p^2}$};    
	\node (Q9) at (7, 4) {$W \times W \times \Qp \times \Qp$};    
    
    \node (Q10) at (0,6) {
   		$W \times W \times W \times W \times \bQ_{p^2} \times \bQ_{p^2}$
    };

   	\draw (Q1)--(Q2);
    \draw (Q1)--(Q3);
    \draw (Q1)--(Q4);
    \draw (Q1)--(Q5);
    \draw (Q2)--(Q6);
    \draw (Q3)--(Q6);
    \draw (Q4)--(Q6);
    \draw (Q6)--(Q10);
    \draw (Q5)--(Q7);
    \draw (Q5)--(Q8);
    \draw (Q5)--(Q9);
    \draw (Q7)--(Q10);
    \draw (Q8)--(Q10);
    \draw (Q9)--(Q10);
    \draw[dotted] (Q2)--(Q7);
    \draw[dotted] (Q3)--(Q8);
    \draw[dotted] (Q4)--(Q9);

    \end{tikzpicture}
    \end{minipage}
    \end{lemma}
    \begin{proof}
    	 Consider $Q, Q'$ quadratic imaginary fields with $p$ inert in $Q$ and $Q'$. There are infinitely many choices for $Q$. Let $B = Q \cdot Q'$, and $B_{\bR}$ be the real quadratic subfield of $B$. We have that $Q \otimes \Qp =Q' \otimes \Qp = \bQ_{p^2}$, hence $B_{\bR}\otimes \Qp = \Qp \times \Qp$. By \Cref{lemma_real_fields}, there are infinitely many totally real fields $R_0$ of degree $g'$, such that the Galois closure $\tilde {R_0}$ is disjoint from $B_{\bR}$, such that $R_0 \otimes \Qp = W \times \Qp$ and with $\Gal(\tilde{R_0}/\bQ) = S_{g'}$. Then $L = Q \cdot Q' \cdot R_0$ will satisfy the assumptions of the lemma. Indeed, $\tilde L = Q \otimes (B_{\bR} \otimes \tilde{R_0})$, so that the Galois group $\Gamma$ can be written as a product
	 \[\Gal(\tilde L/\bQ) = \mu_2 \times \mu_2 \times S_{g'}. \]
	 The subgroup of $\Gamma$ fixing $L$ is $H_L= \{1\} \times \{1\} \times \Fix(1)$, hence the subgroups of $\Gamma$ containing $H_L$ correspond via Galois theory to subfields of $L$. A group theoretic computation gives us the desired diagram. Moreover we see that subfields of $L$ are obtained by taking compositum of $Q, Q', B_{\bR}$ and $R_0$. Hence the diagram $\otimes \Qp$ follows from computations of tensor products of finite $\Qp$-algebras.
    \end{proof}
    According to the diagram, we have $p$-adic valuations $v_1, \overline {v_1}, v_2, \overline{v_2}, w_1, w_2$ on $L$. To be more precise, if $V,\overline V,W$ are the valuations of $F$, then the valuations above $V$ in $L$ are $v_1, v_2$, the valuations above $W$ are $w_1, w_2$. If $V_1, V_2,W$ are the valuations of $F'$, then the valuations above $V_1$ are $v_1, \overline{v_1}$, the valuations above $V_2$ are $v_2, \overline{v_2}$. 
    \begin{lemma} \label{com_abelianscheme}
    		Let $L$ a field satisfying \Cref{nr_diagram}. We can find an abelian scheme $\mathcal A$ with complex multiplication by $L$ and a Weil $q$-number $\pi \in L$ inducing the Frobenius of the special fiber, and such that:
    \begin{equation*}
    		\left\{
    		\begin{array}{l}
    			v_1(\pi) = 0 \\
    			v_2(\pi) = \frac{1}{g'-1} \\
    			w_1(\pi) = \frac{1}{2} \\
    			w_2(\pi) = \frac{1}{2}.
    		\end{array} \right.
    \end{equation*}
   \end{lemma}
   \begin{proof}
   		By \cite[\S 4]{tat71}, we can choose a CM-type $\Phi$ for $L$ such that:
    \begin{equation*}
    \left\{
    \begin{array}{l}
    		\#\Phi_{v_1} = 0 \\
    		\#\Phi_{v_2} = 1 \\
    		\#\Phi_{w_1} = 1 \\
    		\#\Phi_{w_2} = 1.
    	\end{array} \right.
    \end{equation*}
    By \Cref{thm_bourbaki}, we have a $p$-adic field $K$ and an abelian scheme $\mathcal A$ over the ring of integers $\mathcal O_K$, with complex multiplication by $L$ of type $\Phi$. Hence by \Cref{S-T_formula}, we have $\pi \in L$ acting as the Frobenius on the special fibre of $\mathcal A$, with valuations prescribed as in the lemma, given in terms of the CM-type $\Phi$.
   \end{proof}
    
	\begin{proof}[Proof of \Cref{com_case}]
    Let $\mathcal A$ and $\pi$ satisfying \Cref{com_abelianscheme}. For any $k >0$, we have that $v_1(\pi^k) \not = \overline {v}_1(\pi^k)$ and $v_1(\pi^k) \not = v_2(\pi^k)$ hence $\pi^k$ is not in a proper subfield of $L$, and $\bQ(\pi^k) = L$. Hence the special fibre $A$ of $\mathcal A$ is a geometrically simple abelian variety of dimension $g$.
    
    Following \Cref{rmk-rank-eigvalues}, to construct submotives of $A$, we have to exhibit non-trivial relations between the conjugates of the Weil $q$-number $\pi$ associated to $A$. The norm 
    \[N_{L/B}(\pi) \in B\]
is a Weil $q$-number, and $B/\bQ$ satisfies the hypothesis of \Cref{easy_case}. Hence there exists an integer $k > 0$ such that $N_{L/B}(\pi^k)$ is a power of $q$. Moreover, $\pi^k$ is the Frobenius endomorphism of the extension of scalars of $A$ to $\bF_{q^k}$, and there is an extension $K'/K$ of $p$-adic fields such that $\mathcal A':= \mathcal A \times_{\mathcal O_K} \mathcal O_{K'}$ has $A\times_{\bF_q} \bF_{q^k}$ as its special fiber. If we replace $\mathcal A$ by $\mathcal A'$, we are still in the setup of \Cref{nc_abelianscheme}. Hence we can assume that 
\[N_{L/B}(\pi) = q^{g'/2}\]
and by transivity of norms we get
\begin{align*}
	N_{L/Q}(\pi) &= q^{g'}, \\
	N_{L/Q'}(\pi) &= q^{g'}.
\end{align*}
These two relations give rise to exotic submotives $M$ and $M'$ in $\frk h(A)$. We now want to prove that $M$ and $M'$ are the only exotic submotives, hence we have to study multiplicative relations between the conjugates of $\pi$. We can apply \Cref{prop_relation} with $Q:=Q$, $R:= B_{\bR}$, $R':= R_0$ because:
	\begin{itemize}
		\item $B_{\bR}$ and $\tilde R_0$ are linearly disjoint,
		\item $N_{L/B}(\pi) = q^{g'/2}$,
		\item No powers of $N_{L/F}(\pi)$ is in $\bQ$. Indeed $V$, $\overline V$ are $p$-adic valuation on $F$ and 
		\[V(N_{L/F}(\pi)) = v_1(\pi) + v_2(\pi) = \frac{1}{g'-1}\]
		which is not $\overline V(N_{L/F}(\pi))=\frac{2g'-1}{g'-1}$. Hence no powers of $N_{L/F}(\pi)$ is stable by complex conjugation.
		\item $\mu_2$ acts $2$-transitively on $\Hom(B_{\bR}, B_{\bR})$,
		\item $g' \geqslant 2$ and $S_{g'}$ acts $2$-transitively on $\Hom(R_0, \tilde R_0) = \{1, \dots, g'\}$.
	\end{itemize} 
As no power of $\pi$ is in $F$, we are in the first case of \Cref{prop_relation}, and we get that the space of relations between the conjugates of $\pi$ is generated by $N_{L/B}(\pi) = q^{g'/2}$, and $\pi  \overline \pi = q$. But $N_{L/B}(\pi)$ is a product of $g'$ conjugates of $\pi$, and $g'$ is odd, hence the associated motive is of odd weight, and it cannot contain Tate classes. Using the knowledge of a set of generators of the relations between the conjugates of $\pi$, and \Cref{exoset_exomot}, we see that  the only relations contributing to exotic Tate classes are 

\begin{minipage}{0.5 \textwidth}
\begin{align*}
	N_{L/Q}(\pi) &= q^{g'} \\
	N_{L/Q}(\overline \pi) &= q^{g'}
\end{align*}
\end{minipage}
\begin{minipage}{0.5 \textwidth}
\begin{align*}
	N_{L/Q'}(\pi) &= q^{g'} \\
	N_{L/Q'}(\overline \pi) &= q^{g'}.
\end{align*}
\end{minipage}
\vspace{2mm}

Complex conjugation swaps $\pi$ and $\overline \pi$, hence the associated motives $M$ and $M'$ are of rank $2$ and $A$ is mildly exotic.
    \end{proof}
    
    \begin{lemma} \label{lemma_real_fields}
    	Let $p$ be a prime number, $g'$ be an integer and $B_{\bR}$ be a quadratic field, there are infinitely many totally real fields $R_0$ of degree $g^{\prime}$, such that the Galois closure $\tilde {R_0}$ is disjoint from $B_{\bR}$, such that $R_0 \otimes \Qp = W \times \Qp$ and with $\Gal(\tilde{R_0}/\bQ) = S_{g'}$.
    \end{lemma}
    \begin{proof}
    	Let $T$ and $C$ be two prime numbers, different from $p$. Let $\ell$ be another prime number which is inert in $B_{\bR}$ and consider 
	\begin{itemize}
		\item $P_{\bR}$ be a polynomial of degree $g'$ in $\bR[X]$ which has only simple real roots,
		\item $P_{p}$ be the minimal polynomial of a primitive element of $W$ the unramified extension of $\Qp$ of degree $g'-1$,
		\item $S$ a polynomial of degree $g'-2$ which is split with simple roots over $\bQ$,
		\item $P_{\text{transp}}$ a polynomial of degree $2$ which is irreducible over $\bQ_{T}$,
		\item $P_{\text{cycle}}$ a polynomial of degree $g'$ which is irreducible over $\bQ_C$,
		\item $P_\ell$ a polynomial of degree $g'$ which is split with simple roots over $\Ql$.
	\end{itemize}
	By the approximation theorem, we have polynomials $P$ of degree $g'$ in $\bQ[X]$ such that $P$ is arbitrarily close to
	\begin{itemize}
		\item $P_{\bR}$ for the archimedean metric,
		\item $P_p \cdot X$ for the $p$-adic metric,
		\item $S \cdot P_{\text{transp}}$ for the $T$-adic metric,
		\item $P_{\text{cycle}}$ for the $C$-adic metric,
		\item $P_\ell$ for the $\ell$-adic metric.
	\end{itemize}
	 The polynomial $P$ is irreducible because $P_{\text{cycle}}$ is, let $R_0$ be the field $\bQ[X]/(P)$. All roots of $P$ are real because $P$ is close to $P_{\bR}$, hence $R_0$ is totally real. By the Chinese remainder theorem, we get $R_0 \otimes \Qp =W \times \Qp$. The fields $B_{\bR}$ and $\tilde {R_0}$ are disjoint because $\ell$ is inert in $B_{\bR}$ and splits completely in $R_0$, hence in $\tilde {R_0}$. Moreover, it follows that the inclusion 
	\[G' = \Gal(\tilde {R_0}/\bQ) \subset S_{g'}\]
	given by the action of $G$ on conjugates of a primitive element is an equality. Indeed $P$ is close to $S \cdot P_{\text{transp}}$ so that $G'$ contains a transposition, and $P$ is close to $P_{\text{cycle}}$ so that $G'$ contains a cycle of length $g'$. A transposition and a cycle generate $S_{g'}$, hence $G'=S_{g'}$.
	
	Moreover when varying $\ell$, we get infinitely many such $R_0$.
    \end{proof}
\printbibliography
\textsc{T. Agugliaro, IRMA, Strasbourg, France}

\texttt{tagugliaro@math.unistra.fr}

\end{document}